\documentclass[11pt]{amsart}
\usepackage{amsmath,amssymb,mathrsfs, xcolor, geometry}
\usepackage{ stmaryrd }

\newtheorem{theorem}{Theorem}
\newtheorem{proposition}[theorem]{Proposition}

\newtheorem{lemma}[theorem]{Lemma}

\usepackage{enumerate}

\theoremstyle{definition}
\newtheorem{definition}[theorem]{Definition}
\newtheorem{remark}[theorem]{Remark}

\DeclareMathOperator{\spt}{spt}
\DeclareMathOperator{\sing}{sing}
\DeclareMathOperator{\reg}{reg} 

\usepackage{graphicx}

\newcommand{\e}{\operatorname{e}}
\newcommand{\wt}{\widetilde}
\newcommand{\pr}{\partial}

\newcommand{\Lap}{\Delta}

\newcommand{\Ric}{\operatorname{Ric}}

\DeclareMathOperator{\sn}{sn}
\DeclareMathOperator{\cs}{cs}
\DeclareMathOperator{\ct}{ct}
\DeclareMathOperator{\tn}{tn}

\begin{document}

\title[Half-space intersection properties]{Half-space intersection properties for minimal hypersurfaces}
\author{Keaton Naff}
\address{Massachusetts Institute of Technology, Cambridge, MA, USA}
\email{kn2402@mit.edu}
\author{Jonathan J. Zhu}
\address{University of Washington, Seattle, WA, USA}
\email{jonozhu@uw.edu}

\begin{abstract}
We prove ``half-space" intersection properties in three settings: the hemisphere, half-geodesic balls in space forms, and certain subsets of Gaussian space. For instance, any two embedded minimal hypersurfaces in the sphere must intersect in every closed hemisphere. Two approaches are developed: one using classifications of stable minimal hypersurfaces, and the second using conformal change and comparison geometry for $\alpha$-Bakry-\'{E}mery-Ricci curvature. 

Our methods yield the analogous intersection properties for free boundary minimal hypersurfaces in space form balls, even when the interior or boundary curvature may be negative. 

Finally, Colding and Minicozzi recently showed that any two embedded shrinkers of dimension $n$ must intersect in a large enough Euclidean ball of radius $R(n)$. We show that $R(n) \leq 2 \sqrt{n}$.  
\end{abstract}
\maketitle

\section{Introduction}

The classical Frankel property for minimal hypersurfaces asserts that any two closed, embedded, minimal hypersurfaces in the sphere must intersect. This important result, which has since been generalized to other settings, captures a fundamental relationship between minimal embeddings and positive curvature. In this note, we investigate the Frankel property for minimal hypersurfaces in three contexts. In each circumstance, we show that a stronger intersection property holds: embedded minimal hypersurfaces must intersect in every (appropriately interpreted) half-space of the given ambient space.

The first setting we consider is the classical one: the sphere. 

\begin{theorem}\label{thm:hemisphere}
Suppose $\Sigma_0, \Sigma_1 \subset \mathbb{S}^{n+1}$ are closed (compact without boundary)  embedded minimal hypersurfaces of the round sphere and $\mathbb{S}^{n+1}_+ \subset \mathbb{S}^{n+1}$ is any hemisphere. Then 
\begin{equation}
\overline{\mathbb{S}^{n+1}_{+}} \cap \Sigma_0 \cap \Sigma_1 \neq \emptyset.
\end{equation}
\end{theorem}

The next setting we consider is that of free boundary minimal hypersurfaces in geodesic balls in space forms.

\begin{theorem}\label{thm:free-boundary}
Let $M \in \{\mathbb{H}^{n+1}, \mathbb{R}^{n+1}, \mathbb{S}^{n+1}\}$ be a simply-connected space-form which has constant curvature $\kappa \in \{-1, 0, 1\}$ respectively.  Suppose $\Sigma_0, \Sigma_1 \subset B^{n+1}_R \subset M$ are free boundary embedded minimal hypersurfaces in a geodesic ball of radius $R \in (0, \mathrm{diam}(M))$ in the space-form $M$. Moreover, suppose $B^{n+1}_{R,+} \subset B^{n+1}_R$ is any half-ball. Then 
\begin{equation}
\overline{B^{n+1}_{R, +}} \cap \Sigma_0 \cap \Sigma_1 \neq \emptyset.
\end{equation}
\end{theorem}

It is somewhat surprising that this improved Frankel property holds in the presence of nonpositive curvature. In particular, this results holds for geodesic balls in hyperbolic space, which have negative interior curvature. It also holds for large geodesic balls ($R > \frac{\pi}{2}$) in the sphere, which have negative boundary curvature. In general of course, one should not expect a Frankel property to hold in the presence of nonpositive curvature (either in the interior or on boundary). In both cases described above, it seems that there is enough positive curvature (and homogeneity) to overcome the negative curvature regions. We also remark that even the classical Frankel property (asserting an intersection \textit{anywhere} in the ball) has not previously been observed in the presence of negative curvature.

As noted by Assimos \cite{assimos2020intersection}, the half-space intersection properties above immediately imply a \textit{two-piece property}: If $\Sigma$ is a connected (free boundary) minimal hypersurface in $\mathbb{S}^n$ or $B^{n+1}_R$, then any half-space divides $\Sigma$ into two connected components (its intersection with that half-space and the complementary half-space). Two-piece properties were previously proven by Ros \cite{Ros95} and Lima-Menezes \cite{LM21, lima2021two}, and featured in Brendle's classification of genus zero shrinkers \cite{Br16}.

The last context we investigate is that of shrinkers in Euclidean space. We find an explicit radius $2\sqrt{n}$ for Colding-Minicozzi's `strong Frankel property' for self-shrinkers \cite{CM23}:

\begin{theorem}\label{thm:shrinkers}
Suppose $\Sigma_0, \Sigma_1 \subset \mathbb{R}^{n+1}$ are complete, properly embedded $n$-dimensional shrinkers of Euclidean space. Then 
\begin{equation}
\overline{B^{n+1}_{2\sqrt{n}}} \cap \Sigma_0 \cap \Sigma_1 \neq \emptyset.
\end{equation} 
\end{theorem}

The reader may note that Theorem \ref{thm:shrinkers} is not quite a true half-space setting, as the sphere of radius $2\sqrt{n}$ is not itself a shrinker. It is natural to wonder whether the intersection property holds for complete embedded shrinkers inside the sphere of radius $\sqrt{2n}$ (which \textit{is} a shrinker). However, the radius $2\sqrt{n}$ appears to be sharp for any method which utilises the inside of the sphere only (see Remark \ref{rem:shooting-method} below). A genuine half-space property for shrinkers indeed holds: two complete embedded $n$-dimensional shrinkers in $\mathbb{R}^{n+1}$ must intersect in any halfspace $\mathbb{R}^{n+1}_+ \subset \mathbb{R}^{n+1}$. This was proven recently in work of Choi-Haslhofer-Hershkovits-White \cite{CHHW22}, so we will not discuss its proof here, although it also follows from our methods.

The classical Frankel property has at least three well-known proofs, respectively using:
\begin{enumerate}
\item Bernstein theorems, which classify stable minimal hypersurfaces; 
\item Length variation (applied either to a connecting geodesic, or the distance to a minimal hypersurface);
\item Reilly's formula. 
\end{enumerate}

In \cite{assimos2020intersection}, Assimos proposed an alternative approach to proving Theorem \ref{thm:hemisphere}. 
In this note, however, we will adapt the first two classical methods above. 

For method (1), the first step is to establish a suitably general Bernstein result in the appropriate half-space for varifolds (Propositions \ref{prop:current-bernstein}, \ref{prop:fb-current-bernstein}, and \ref{prop:gaussian-current-bernstein} below). The key is that the spaces involved support only a very limited collection of stable minimal hypersurfaces; in particular, if two minimal hypersurfaces do not intersect, then solving a certain Plateau problem on the region between will detect a contradiction. In a half-space context, this proof strategy has previously been used in the Gaussian setting \cite{CM23, CHHW22} and also in the free boundary, Euclidean setting \cite{LM21, lima2021two}.

For method (2), we consider the distance function to a minimal hypersurface. Classically, this distance function is superharmonic for minimal hypersurfaces without boundary, in an ambient space of nonnegative Ricci curvature. In the cases that interest us, where the hypersurfaces have boundary, we will need to make a conformal change to adapt this technique. The conformal change blows the fixed boundary to infinity and, crucially, has \textit{Bakry-\'{E}mery-}Ricci curvature which enjoys a sufficient positivity condition. This approach appears to a descendent of Ilmanen's barrier principle and ``moving around barriers'' (\cite{Ilm96}). We note that certain Frankel properties were proven in the positive Bakry-\'{E}mery-Ricci setting by Wei--Wylie \cite{WW09} and Moore-Woolgar \cite{MW21}; our results in this regard are slightly more general. 

We remark that our second approach (2) seems to fit into a broader theme of exploiting conformal and weighted techniques in the study of minimal hypersurfaces, (various forms of) nonnegative curvature, and their relationship. We highlight, for instance, the conformal metric introduced by Fischer-Colbrie \cite{FC85} in the study of minimal surfaces in 3-manifolds with finite index, or the weighted minimal slicings used by Schoen and Yau in their proof positive mass theorem \cite{SY79}. Recently, nonnegative $\alpha$-Bakry-\'Emery-Ricci played a crucial role in one proof of the stable Bernstein theorem for minimal immersions into $\mathbb{R}^4$ \cite{CMR22}. A conformal change that yields positive scalar curvature played a crucial role in another proof \cite{CL23} (see also \cite{CL24} for the first proof).  

In the parabolic setting, a similar conformal change can be used to prove Ilmanen's localized avoidance principle for mean curvature flow, as detailed in work of Chodosh-Choi-Mantoulidis-Schulze \cite[Appendix C]{CCMS20}. The test function used in \cite{CCMS20} is also similar to the test function we use for Theorem \ref{thm:shrinkers}, and indeed these methods relate back to Ilmanen's concept of `moving around barriers' \cite{Ilm96} for minimal hypersurfaces. In particular, there may be a proof of Theorem \ref{thm:shrinkers} that uses mean curvature flow by exploiting this similarity, however we do not pursue it here. 

We believe that it would be interesting to find a proof of half-space Frankel properties using the Reilly formula method (3), but we do not address it in this note. 

\subsection{Some more general results} 

In each of the settings above, we need only work with hypersurfaces (properly embedded) inside the half-space in question. Specifically, we will show the somewhat stronger results: 

\begin{theorem}\label{thm:hemisphere-more-general}
Let $\mathbb{S}^{n+1}_+$ be a round hemisphere and suppose $\Sigma_0, \Sigma_1 \subset \mathbb{S}^{n+1}_+$ are properly embedded $(\partial \Sigma_i \subset \partial \mathbb{S}^{n+1}_+)$  smooth minimal hypersurfaces. If the boundaries do not intersect, $\partial \Sigma_0 \cap \partial \Sigma_1 = \emptyset$, then the minimal hypersurfaces must intersect in the interior, $\Sigma_0 \cap \Sigma_1 \neq \emptyset$.
\end{theorem}

\begin{theorem}\label{thm:free-boundary-more-general}
Let $M \in \{\mathbb{H}^{n+1}, \mathbb{R}^{n+1}, \mathbb{S}^{n+1}\}$ be a simply-connected space-form which has constant curvature $\kappa \in \{-1, 0, 1\}$ respectively. Let $B^{n+1}_{R, +} \subset M$ denote a geodesic half-ball of radius $R \in (0, \mathrm{diam}(M))$. Then $\partial B^{n+1}_{R, +} = D \cup S$ where $D = \overline{\partial B^{n+1}_{R, +} \setminus \partial B^{n+1}_R}$ is the closed totally geodesic portion of the boundary and $S = \partial B^{n+1}_{R, +} \setminus D$. 

Suppose $\Sigma_0, \Sigma_1 \subset B^{n+1}_{R, +}$ are properly embedded $(\partial \Sigma_i \subset \partial B^{n+1}_{R, +})$ minimal hypersurfaces in $B^{n+1}_{R, +}$ which meet the piece of the boundary $S$ orthogonally. If the boundaries do not intersect, $\partial \Sigma_0 \cap \partial \Sigma_1 = \emptyset$, then the minimal hypersurfaces must intersect in the interior, $\Sigma_0 \cap \Sigma_1 \neq \emptyset$. 
\end{theorem}

\begin{theorem}\label{thm:ball-shrinker-more-general}
Suppose $\Sigma_0, \Sigma_1 \subset B_{2\sqrt{n}}^{n+1} \subset \mathbb{R}^{n+1}$ are properly  embedded $(\partial \Sigma_i \subset \partial B^{n+1}_{2\sqrt{n}})$ smooth hypersurfaces satisfying the shrinker equation. If the boundaries do not intersect, $\partial \Sigma_1 \cap \partial \Sigma_2 = \emptyset$, then the shrinkers must intersect in the interior, $\Sigma_0 \cap \Sigma_1 \neq \emptyset$.
\end{theorem}

Theorems 1, 2, and 3 will readily follow from Theorems 4, 5, and 6 by a continuity argument and the transversality theorem. 

\begin{figure}
\includegraphics[scale=0.37]{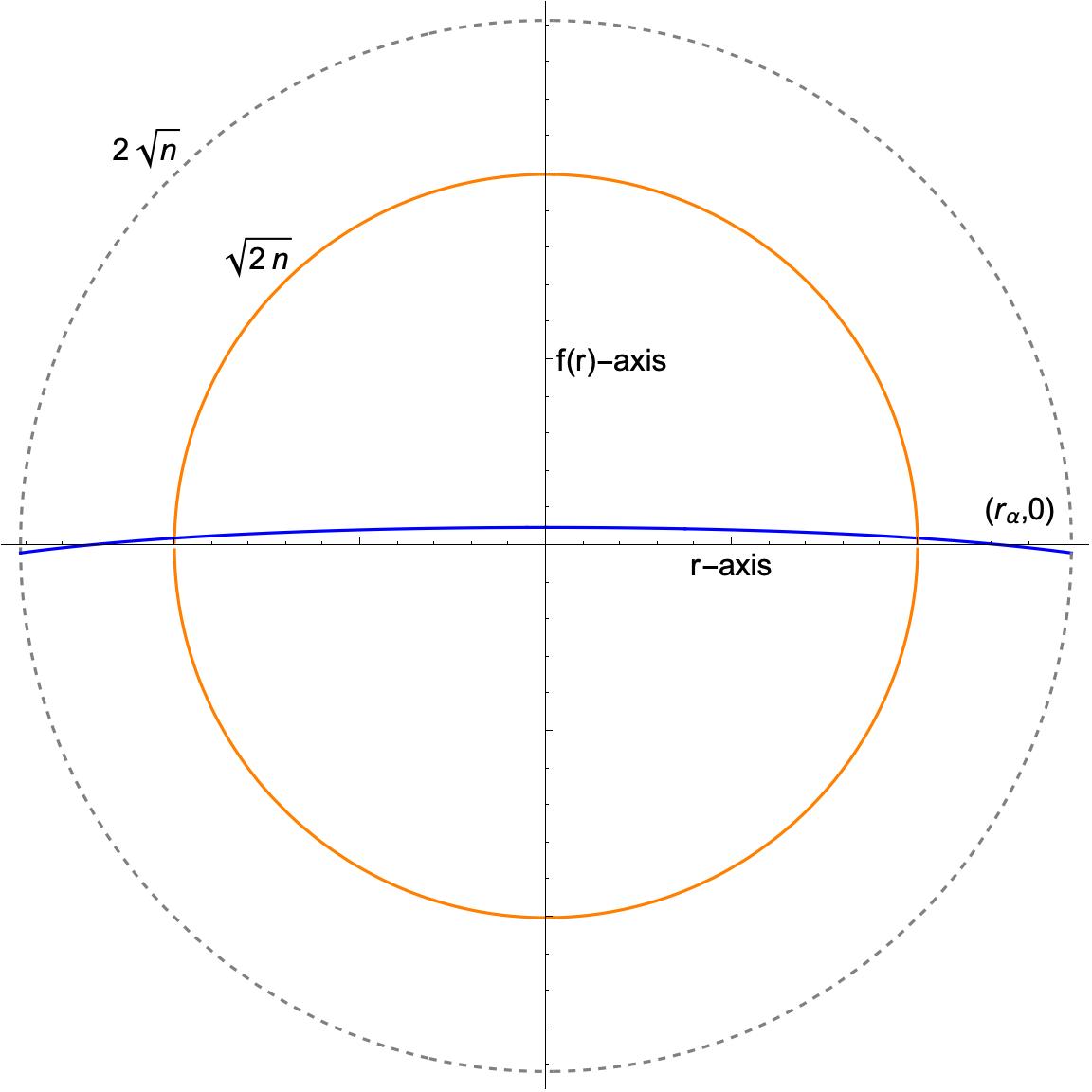}
\caption{Numerical solution of the $n$-dimensional rotationally symmetric shrinker equation $\frac{f''(r)}{(1+f'(r)^2)} - \big(\frac{r}{2} -\frac{n-1}{r}\big)f'(r) + \frac{1}{2}f(r) = 0$, with $f(0)=f_0>0$, $f'(0)=0$. Here $n=2$, and the initial conditions are simulated by taking initial point $(r_0,f_0)= (10^{-6},10^{-1})$ and direction $\theta(10^{-6}) = 0$ (reflecting for $r<0$). Denoting the positive $r$-intercept by $(r_\alpha, 0)$, the portion of this curve inside $B_R$ lies strictly in the upper half plane for $R<r_\alpha$. Moreover, as $f_0 \to 0$ it appears that $r_\alpha\to 2\sqrt{n}$. Circles of radii $\sqrt{2n}$ (filled in orange) and $2\sqrt{n}$ (dashed in grey) included for reference.}\label{fig:shooting-method}
\end{figure}

\begin{remark}\label{rem:shooting-method}
Numerical simulations (shooting method, see Figure \ref{fig:shooting-method}) for the rotationally symmetric shrinker equation suggest that for any $R<2\sqrt{n}$, there is a curve $\gamma$ whose rotation satisfies the shrinker equation, is properly embedded in $B_R^{n+1}$, and is contained in a strict half-space. In particular, its reflection $\bar{\gamma}$ across the rotational axis will also define a properly embedded shrinker in $B_R$, disjoint from the original. This would give a counterexample to a strong Frankel property for properly embedded shrinkers in Euclidean balls of radius $R < 2\sqrt{n}$. We expect that analysis similar to \cite[Section 2.3]{Dr15} may establish this counterexample rigorously.
\end{remark}

\subsection{Paper outline} 
In Section 2, we establish preliminaries on stability in all three settings and establish some basic results needed for space forms. In Section 3, we establish the generalized Bernstein result in the hemisphere and prove Theorems 1 and 4. In Section 4, we do the same in the free boundary setting and prove Theorems 2 and 5. In Section 5, we sketch the adaptations necessary to prove Theorems 3 and 6 via the Bernstein and Plateau proof method (1). In Section 6, we give alternative proofs of our results based on the length variation method (2) and a differential inequality satisfied by the distance function to minimal hypersurfaces.

\subsection*{Acknowledgements}
KN was supported by the National Science Foundation under grant DMS-2103265. JZ was supported in part by a Sloan Research Fellowship. The authors wish to thank Giada Franz for helpful comments on a previous draft.

\section{Preliminaries}
\label{sec:prelims}

Throughout this section, we suppose $(M, g)$ is a complete $(n+1)$-dimensional Riemannian manifold and $N \subset M$ is an open connected subset. We let $S \subset \partial N$ denote a smooth manifold (open in $\partial N$). We allow $S = \emptyset$ and let $D := \partial N \setminus S$ denote the closed complement of $S$ in $\partial N$. We will sometimes call $D$ the \textit{fixed-boundary} and $S$ the \textit{free boundary} in $\partial N$. 

For our main results, $M, N, D, S$ will belong to the following list: 
\begin{enumerate}
\item[(I)] A hemisphere $N =\mathbb{S}^{n+1}_+$ inside the (round) sphere $M = \mathbb{S}^n$ with $D = \partial N$ and $S = \emptyset$. 
\item[(II)] A geodesic half-ball $N = B^{n+1}_{R, +}$ (see Section \ref{sec:test-function}) of radius $R$ in a space form $M \in \{\mathbb{H}^{n+1}, \mathbb{R}^{n+1}, \mathbb{S}^{n+1}\}$ of curvature $\kappa \in \{-1, 0, 1\}$ with $D = \overline{\partial B^{n+1}_{R, +} \setminus \partial B^{n+1}_R}$ the closed totally geodesic portion of the boundary and $S = \partial B^{n+1}_{R, +} \setminus D$ the open complement of $D$.
\item[(III)] $M=(\mathbb{R}^{n+1},g)$ has the Gaussian metric $g = e^{-|x|^2/2n}g_{\mathbb{R}^{n+1}}$, and $N$ is the ball $N = B^{n+1}_{2\sqrt{n}} \subset \mathbb{G}^{n+1}$ of (Euclidean) radius $2\sqrt{n}$, with $D = \partial N$ and $S = \emptyset$. 
\end{enumerate}

We are primarily interested in properly embedded hypersurfaces in $N$ with boundary; when $S$ is nonempty we will impose free boundary conditions along $S$. To make this notion precise we take the following definition:

\begin{definition}
\label{def:half-fb}
Let $M, N, S, D$ be as above (allowing $S = \emptyset$). We say a hypersurface $\Sigma \subset N$ is a properly embedded, $S$ free boundary hypersurface if:
\begin{enumerate}[(i)]
\item $(\Sigma, \pr \Sigma) \hookrightarrow (N,\pr N)$ is a smooth, proper embedding. 
\item $\Sigma$ meets $S$ orthogonally; that is, at any boundary point $x \in \partial \Sigma \cap S$, the (outer) conormal of $\partial \Sigma$ in $\Sigma$ is equal to the (outer) normal of $S$ in $N$.  
\end{enumerate}
\end{definition}

Our goal is to study the intersection properties of properly embedded, $S$ free boundary minimal hypersurfaces in $(N, g)$ as above.

\subsection{Stability}
\label{sec:stability}

Let $M, N, D, S, \Sigma$ be as before. In this section, we adopt the convention that $\Sigma \subset \overline{N}$ is relatively open in $\overline{N}\setminus D$; that is $\Sigma$ includes its boundary in $S = \pr N\setminus D$ or equivalently $\Sigma = \overline{\Sigma} \cap (\overline{N} \setminus D)$. 

A variation $\Sigma_t = \phi_t(\Sigma)$ of $\Sigma=\Sigma_0$ ($\phi_0 = \mathrm{Id}$) (more generally, of a current or varifold) is the image under a 1-parameter family of diffeomorphisms $\phi_t$ of $N$; we denote the generator by $X = \partial_t \phi_t |_{t=0}$. Note that $X$ need not vanish on $\partial N$. 

A properly embedded hypersurface $\Sigma$ in stable in $(N,\pr N)$ if it is minimal and, for any $v\in C^\infty_0(\Sigma)$, 
\begin{equation}
\label{eq:usual-stability}
0\leq -\int_\Sigma v L_\Sigma v , \qquad L_\Sigma:= \Lap_\Sigma + |A_\Sigma|^2 + \Ric_M (\nu,\nu).
\end{equation}
Here $\nu$ is a choice of unit normal along $\Sigma$. The right hand side is precisely the second variation of area under the variation generated by $X = v\nu$.

More generally, if $S$ may be nonempty, let $\eta$ denote its outward-pointing normal in $N$ and $A_S$ denote its second fundamental form (with respect to $\eta$). A properly embedded, $S$ free boundary hypersurface $\Sigma$ is \textit{stable in $(N, D)$} if it is minimal and for any $v\in C^\infty_0(\Sigma)$, 
\begin{equation}
\label{eq:fb-stability}
0\leq -\int_\Sigma v L_\Sigma v + \int_{\pr \Sigma \cap S} v(\pr_\eta - A_S(\nu,\nu))v.
\end{equation}
Note that the $S$ free boundary condition implies that $X= v\nu$ is tangent along $S$.

By integration by parts and standard approximation arguments, if $\Sigma$ is stable in $(N,D)$, then for all $v\in C^{0,1}_0(\Sigma)$ we have the gradient stability inequality
\begin{equation}
\label{eq:fb-stability-gradient}
\int_{\Sigma} (|A_{\Sigma}|^2 + \mathrm{Ric}_M(\nu ,\nu)) v^2 \leq \int_{\Sigma} |\nabla^{\Sigma} v|^2 + \int_{\partial \Sigma \cap S} v(\pr_\eta - A_S(\nu,\nu))v.
\end{equation}
Again when $S = \emptyset$, the second term on the right hand side may be taken to be zero.

Finally, we say that $\Sigma$ is stable \textit{with respect to one-sided variations} if the corresponding stability inequalities hold under the additional assumption $v\geq 0$.

\begin{remark}[Unstable deformations in the full space] 
\label{rmk:instability}
We remark that geodesic balls $B^{n+1}_R \subset M$ in space forms do not have any stable free boundary minimal hypersurfaces (cf. \cite{Z23}). It will be useful to obtain mean convex deformations of these minimal surfaces, which (as in the closed setting) one expects can be obtained by moving in the direction of the first eigenfunction. However, in Section \ref{sec:bernstein-fb-full} we will only be interested in mean convex deformations of totally geodesic free boundary hypersurfaces in $B^{n+1}_R$. We find it somewhat easier to deal with this case by hand using conformal maps (see Appendix \ref{sec:conformal-deformations}).
\end{remark}


\subsection{Space forms, totally geodesic hypersurfaces, and smooth Bernstein results}
\label{sec:test-function}

In this section, we observe that the distance from a totally geodesic hypersurface provides a useful test function for the stability inequality in space forms.

Given a simply-connected space form $M \in \{\mathbb{H}^{n+1}, \mathbb{R}^{n+1}, \mathbb{S}^{n+1}\}$ of constant curvature $\kappa \in \{-1, 0, 1\}$, we define 
\begin{equation}\label{def:sn}
\sn(r) := \begin{cases} \sinh(r),& M= \mathbb{H}^{n+1}\\  r,& M= \mathbb{R}^{n+1}\\ \sin(r),& M= \mathbb{S}^{n+1}\end{cases}.
\end{equation}
and set $\cs(r) := \sn'(r)$, as well as $\tn(r) := \sn(r)/\cs(r)$ and $\ct(r) :=1/\tn(r)$. Given a radius $R > 0$, we let $B^{n+1}_R\subset M$ denote a geodesic ball of radius $R$ with a center $o \in M$. We let $B^{n+1}_{R,+}$ denote a geodesic ``half-ball". That is, if $M' \subset M$ is a totally geodesic complete hypersurface passing through $o$, then $M'$ separates $M$ into two (identical) open subsets $M_+ \cup M_-$ and we take
\begin{equation}
B^{n+1}_{R, +} := B^{n+1}_R \cap M_+. 
\end{equation}

\begin{lemma}\label{lem:space-form-test-function}
Let $M \in \{\mathbb{H}^{n+1}, \mathbb{R}^{n+1}, \mathbb{S}^{n+1}\}$ be a simply-connected space-form which has constant curvature $\kappa \in \{-1, 0, 1\}$ respectively. Let $M' \subset M$ be any complete totally geodesic hypersurface. Let $\rho: M \to [0, \infty)$ be the distance function to $M'$. Then 
\begin{equation}
\nabla^2 \sn(\rho) = -\kappa \sn(\rho) g.
\end{equation}
Moreover, let $o\in M'$ and $r$ be the distance function from $o$. Then 
\begin{equation}
\langle \nabla \sn(\rho), \nabla r\rangle = \ct(r)\sn(\rho). 
\end{equation}
\end{lemma}

\begin{proof}
It is straightforward to verify that the metric on a space form can be written inductively as $g = d\rho^2 + \cs(\rho)^2 g'$ on $M \setminus M'$. In particular, since $\nabla \rho = \partial_\rho$, we have 
\[
\nabla^2 \rho = \frac{1}{2} \mathcal{L}_{\partial \rho} g = -\kappa \cs(\rho) \sn(\rho) g' =  -\kappa \tn(\rho) (g - d\rho\otimes d\rho).
\]
The formula for $\nabla^2 \sn(\rho)$ readily follows. 

Consider $x\in M$, let $z\in M'$ be the nearest point projection and $r_z$ the distance function from $z$. Then at $x$ we have $\nabla \rho  = \nabla r_z$, and the law of cosines on the right triangle $ozx$ gives (for $\kappa \neq 0$)
\[ \langle \nabla \rho, \nabla r\rangle = \langle \nabla r_z,\nabla r\rangle= \kappa \frac{\frac{\cs r}{\cs \rho} - \cs r \cs \rho}{\sn r \sn \rho} = \frac{\ct r}{\sn \rho \cs \rho} \kappa (1- \cs^2 \rho) = \ct r \tn \rho .\]
When $\kappa=0$, the law of sines immediately gives $\langle \nabla \rho, \nabla r\rangle =  \langle \nabla r_z,\nabla r\rangle= \frac{\rho}{r} = \ct r \tn \rho$. In either case, this implies the result. 
\end{proof}

In particular, on a minimal hypersurface $\Sigma\subset M$, it follows from the lemma that 
\begin{equation}\label{eq:space-form-test-function}
(\Lap_\Sigma + n\kappa) \sn \rho = 0 \;\; \text{ and } \;\;(\pr_r - \ct(r)) \sn \rho =0.
\end{equation}
Note that a space form has  $\Ric_M(\nu, \nu) = n\kappa$, and any geodesic sphere $S=\pr B^{n+1}_R \subset M$ has second fundamental form $A_S(\nu, \nu) = \ct(R)$.
Thus, 
\begin{equation} 
L_\Sigma \sn \rho = |A_\Sigma|^2 \sn \rho \;\; \text{ and } \;\; (\partial_\eta - A_S(\nu, \nu)) \sn \rho\big|_{\partial \Sigma \cap S} = 0.
\end{equation} 
Thus if $\Sigma$ is stable in settings (I) or (II), the stability inequality \eqref{eq:fb-stability} yields $\int_\Sigma |A_\Sigma|^2 \sn^2\rho \leq 0$, which forces $\Sigma$ to be totally geodesic. As $\sn \rho$ vanishes on the totally geodesic hypersurface $M'$, we have the following classifications of stable minimal hypersurfaces to one side of $M'$.

\begin{proposition}\label{prop:smooth-bernstein}
Suppose that $\Sigma^n \subset \mathbb{S}^{n+1}_+$ is a properly embedded minimal hypersurface of a hemisphere. Then $\Sigma$ is stable if and only if $\Sigma$ is totally geodesic. 
\end{proposition}

\begin{proposition}\label{prop:fb-smooth-bernstein}
Let $M \in \{\mathbb{H}^{n+1}, \mathbb{R}^{n+1}, \mathbb{S}^{n+1}\}$ be a simply-connected space-form. Suppose that $\Sigma^n \subset B^{n+1}_{R, +}$ is a properly embedded minimal hypersurface of a geodesic half-ball in $M$ of radius $R \in (0, \frac{1}{2}\mathrm{diam}(M))$ which meets $S = \partial B^{n+1}_{R, +} \cap M_+$ orthogonally. Then $\Sigma$ is stable if and only if $\Sigma$ is totally geodesic. 
\end{proposition}

To study the intersection properties, we will need the versions of these results which also hold for \textit{singular} minimal hypersurfaces - the reader may consult Sections \ref{sec:bernstein-sphere-formal} and \ref{sec:bernstein-fb-formal} for the formal, generalised proofs. 

We remark that Proposition \ref{prop:smooth-bernstein} was also previously observed by Choe-Soret \cite{CS09}.

\subsection{Currents and varifolds}

In this article, our proofs of the Frankel properties will rely on solving certain variational problems. As such, we briefly review some basic facts we will need from geometric measure theory. We refer the reader to \cite{Si83} for additional details and background. 

In what follows, an $n$-current will always mean an integer multiplicity, rectifiable $n$-current. We let $T \in \mathbb{I}_n(N, \partial N)$ denote an $n$-current in $M$ such that its support lies in $\overline{N}$ and its boundary $\partial T$ is an $(n-1)$-current supported in $\partial N$. We let $\mathcal{H}^n$ denote the $n$-dimensional Hausdorff measure. Given $T$, we let $\theta$ denote its integer-valued multiplicity function, let $\mu_T = \mathcal{H}^n \lfloor\, \theta$ denote its mass measure, and $\spt T := \spt \mu_T$ denote its support. 

The support of the current $T$ can be decomposed into the set of regular points and the set of singular points. Here $\reg T \subset \spt T$ consists of any point $x \in \spt T$ such that either: 
\begin{itemize}
\item (interior regular point) $x \in N \cap \spt T$ there exists an $n$-dimensional, oriented $C^1$ manifold $\Sigma \subset N$ and an integer $m$ such that $T = m [[\Sigma]]$ in a neighborhood of $x$, or else
\item (boundary regular point) $x \in S \cap \spt T$ and there exists an $n$-dimensional, oriented $C^1$ manifold $\Sigma$ with $C^1$ boundary $\partial \Sigma \subset S$ meeting $S$ orthogonally and an integer $m$ such that $T = m[[\Sigma]]$ in a neighborhood of $x$ (in $\overline{N}$). 
\end{itemize}
Of course, when $S$ is empty, we will only have interior regular points. Further, let 
\[
\sing T := \spt T \setminus (\reg T \cup (\spt \partial T \cap D))
\]
 denote the \textit{relative} set of singular points.

Given an $n$-current $T$, we let $V = |T|$ denote the associated $n$-rectifiable integer multiplicity varifold. For $\mu_T = \mu_V$ a.e. $x \in \spt T$, we let $T_xV$ denote the approximate tangent space $V$ at $x$. The first variation of an $n$-varifold $V$ is given by 
\[
\delta V(X) := \int \mathrm{div}_{V} X \, d\mu_V 
\]
where, for any $x \in \spt V$ for which $T_xV$ exists, we have
\[
\mathrm{div}_V X(x) := \mathrm{div}_{T_xV}X (x) = \sum_{i =1}^n \langle \nabla_{e_i} X, e_i \rangle (x).
\]
Here $\nabla$ denotes the connection on $M$ and $e_1, \dots, e_n$ is any orthonormal basis for $T_xV$. 

For $D \subset \partial N$ a closed subset and $S = \partial N \setminus D$ smooth, we will say a varifold $V$  is \textit{stationary relative to $D$} or stationary in $(N, D)$ if $\delta V(X)=0$ for all $C^1$ vector fields $X$ which have compact support in $\overline{N} \setminus D$ and are tangent along $S$. A current $T$ is stationary if the associated varifold $V = |T|$ is.

\subsection{A cutoff lemma}

In order to justify the stability inequality for our test functions on singular minimal hypersurfaces, we will need a cutoff lemma. Its proof which is mostly standard, can be found in Appendix \ref{sec:cutoff}.

\begin{lemma}\label{lem:cutoff}
With $M, N, S, D$ as before, for $\delta \geq 0$, let $N_\delta = \{x \in N : d(x, D) \geq \delta \}$. Suppose $T \in \mathbb{I}_{n}(N, \partial N)$ is an $n$-current which is stationary in $(N, D)$. Assume the regular set $\Sigma = \reg T$ has finite area $\mathcal{H}^n(\Sigma) < \infty$ and assume the (relative) singular set $\sing T$ satisfies the Hausdorff dimension estimate $\mathrm{dim}_{\mathcal{H}} \sing T < n - 2$. 

Then, for any $\delta > 0$ and $\varepsilon > 0$, there exists a Lipschitz function $\phi : N \to [0, 1]$ which vanishes in a neighborhood of $\sing T \cap N_\delta$ and satisfies 
\[
\int_{\Sigma \cap N_\delta} |\nabla^{\Sigma} \phi|^2   \leq \varepsilon.
\]
\end{lemma}

We remark that, for minimising currents, well-known regularity theory (\cite{Fed70},\cite{GR87}) ensures that the singular set has Hausdorff dimension at most $n -7$. 

\section{Bernstein theorems and Frankel properties in the (hemi-)sphere}

In this section, we give a proof of the improved Frankel property for embedded minimal hypersurfaces in the (hemi-)sphere. 

\subsection{Bernstein result for stable currents in a hemisphere}
\label{sec:bernstein-sphere-formal}

We now prove the analogue of Proposition \ref{prop:smooth-bernstein} for stable currents. 

\begin{proposition}\label{prop:current-bernstein}
Suppose that $T \in \mathbb{I}_n(\mathbb{S}^{n+1}_+, \partial \mathbb{S}^{n+1}_+)$ is an integral $n$-current in the hemisphere which is stationary in $(\mathbb{S}^{n+1}_+, \partial \mathbb{S}^{n+1}_+)$. 
Assume the regular set $\Sigma = \reg T$ has finite area $\mathcal{H}^n(\Sigma) < \infty$ and assume the singular set $\sing T = \spt T \setminus (\Sigma \cup \spt \partial T)$ of $T$ satisfies the Hausdorff dimension estimate $\mathrm{dim}_{\mathcal{H}} \sing T < n - 2$. If $\Sigma$ is stable in $(\mathbb{S}^{n+1}_+, \pr \mathbb{S}^{n+1}_+)$ with respect to one-sided variations, then its second fundamental form $A_{\Sigma}$ vanishes everywhere.
\end{proposition}

\begin{proof}
For ease of notation let $N = \mathbb{S}^{n+1}_+$ and let $\rho$ be the distance to the totally geodesic boundary $\pr N$, so that $N = \{\rho >0\}$. Let $N_\delta = \{x \in N : d_N(x, \partial N) \geq \delta\}$ as before.  Let $\Sigma = \reg T$ and let $u := \sn \rho |_{\Sigma}> 0$. 

We claim that, for any nonnegative Lipschitz function $v$ with compact support in $\Sigma$, we will have 
\begin{equation}
\label{eq:current-bernstein}
\int_\Sigma v^2 |A_\Sigma|^2 u^2 \leq  \int_\Sigma u^2 |\nabla^{\Sigma} v|^2. 
\end{equation}
To prove the claim, note that by approximation and dominated convergence, it suffices to consider smooth $v$. By the classical computation on the regular part $\Sigma$ we have $L_\Sigma u = |A_\Sigma|^2 u$ (see Section \ref{sec:test-function}). Then $uv\geq 0$ is a valid test function for the stability inequality \eqref{eq:usual-stability}, which gives
\begin{align*}
0&\leq -\int_\Sigma uv L_\Sigma(uv)   \\
& = -\int_\Sigma \left( v^2 uL_\Sigma u +  2uv \langle \nabla u, \nabla v\rangle + u^2 \phi \Delta_{\Sigma} v \right) \\
& =  -\int_\Sigma \left( v^2 |A_\Sigma|^2 u^2 + \mathrm{div}_{\Sigma}( u^2 v \nabla^\Sigma v) - u^2 |\nabla^\Sigma v|^2  \right) \\
& =  -\int_\Sigma \left( v^2 |A_\Sigma|^2 u^2 -  u^2 |\nabla^{\Sigma} v|^2 \right).
\end{align*}
Moving the first term to the left gives \eqref{eq:current-bernstein}, and establishes the claim.

Consider $\delta, \varepsilon > 0$ small. Let $\psi = \psi_\delta : N \to [0,1]$ denote a smooth cutoff function on $N$ such that $\psi \equiv 1$ in $N_{2\delta}$, $\psi \equiv 0$ on $B_\delta(\partial N)$, and $|\nabla \psi| \leq C\delta^{-1}$.  Let $\phi = \phi_{\varepsilon}^\delta$ denote the cutoff function given by Lemma \ref{lem:cutoff}. Then $v := \psi \phi \in C^{0, 1}_0(\Sigma)$. Indeed, it has compact support on $\Sigma$ since $\psi$ vanishes outside $N_\delta \cap \Sigma$ and $\phi$ vanishes in a neighborhood of $\sing T \cap N_\delta$.

In particular, \eqref{eq:current-bernstein} holds for $v= \psi \phi$. 
Noting that
\[
 |\nabla^{\Sigma} v|^2 \leq 2\psi^2 |\nabla^{\Sigma} \phi|^2 + 2 \phi^2|\nabla^{\Sigma} \psi|^2,
\]
we obtain 
\[
\int_{\Sigma} |A_{\Sigma}|^2 u^2 \psi^2 \phi^2   \leq 2 \int_{\Sigma} u^2\psi^2 |\nabla^{\Sigma} \phi|^2 +2 \int_{\Sigma} u^2 \phi^2 |\nabla^{\Sigma} \psi|^2.
\]
Since $u^2 \psi^2 \leq 1$, the gradient estimate in Lemma \ref{lem:cutoff} gives 
\[
\int_{\Sigma} |A_{\Sigma}|^2 u^2 \psi^2 \phi^2   \leq 2 \varepsilon + 2 \int_{\Sigma} u^2\phi^2 |\nabla^{\Sigma}\psi|^2.
\]
On the other hand, for the remaining integral on the right hand side, we have $\phi^2 \leq 1$ and the function $|\nabla^{\Sigma} \psi|^2$ is only supported on $\Omega_\delta := N_{2\delta} \setminus B_{\delta}(\partial N)$.  In this region, we have $|\nabla^{\Sigma} \psi| \leq C \delta^{-1}$. But since $u$ satisfies $u |_{\partial N} = 0$, we have $\sup_{\Omega_\delta} u \leq C \delta$ for some $C$ independent of $\delta$. Consequently, 
\[
\sup_{\Omega_\delta \cap \Sigma}u^2 \phi^2 |\nabla^{\Sigma} \psi|^2 \leq C.
\]
We conclude  
\[
\int_{\Sigma} |A_{\Sigma}|^2 u^2 \psi^2 \phi^2 \leq 2 \varepsilon + C \mathcal{H}^{n}(\Sigma \cap \Omega_\delta). 
\]
We first send $\varepsilon \to 0$ and then $\delta \to 0$. By Fatou's lemma, we obtain  
\[
\int_{\Sigma } |A_{\Sigma}|^2 u^2  \leq 0. 
\]
Since $u > 0$ on $\Sigma$, this implies $A_{\Sigma} \equiv  0$, as claimed. 

\end{proof}

\subsection{Proof of Theorem \ref{thm:hemisphere-more-general}}

In this section, we establish the improved Frankel property stated in Theorem \ref{thm:hemisphere-more-general}. The guiding principle is that, by the Bernstein result in the previous section, the hemisphere only supports a limited class of stable minimal hypersurfaces. This may be quantified by their boundaries, so we will exploit the Bernstein result by constructing a suitable boundary for the Plateau problem using the original minimal hypersurfaces as barriers. 

For $i = 0,1$, consider $\Sigma_i \subset N:= \mathbb{S}^{n+1}_+$ which are smooth proper embedded minimal hypersurfaces in the hemisphere, with smooth embedded boundary $\Gamma_i \subset \partial N$. We may assume without loss of generality that $\Sigma_i$ is connected. Note that, by the usual Frankel property in $\mathbb{S}^{n+1}$ (since $\partial N$ is minimal), each $\Gamma_i$ must be nonempty. 

Under these preliminary assumptions, we now give a proof of Theorem \ref{thm:hemisphere-more-general}. After possibly relabelling, it suffices to consider three cases: 
\begin{itemize}
\item Case 1: Both $\Sigma_i$ are totally geodesic;

\item Case 2: $\Sigma_0$ is totally geodesic, but $\Sigma_1$ is not; or

\item Case 3: Neither $\Sigma_i$ is totally geodesic. 
\end{itemize}

In Case 1, each $\Gamma_i$ must be totally geodesic in $\pr N$, and hence intersect (e.g. by the usual Frankel property). Thus we may assume that we are in Case 2 or Case 3. We suppose for the sake of contradiction that $\overline{\Sigma}_0 \cap \overline{\Sigma}_1 = \emptyset$. 

Our first step is to describe a suitable domain and a suitable boundary for the Plateau problem. The hypersurfaces $\Sigma_0, \Sigma_1$ are both separating in the hemisphere $N$. Let $\Omega^{\pm}_i$ be the connected components of $N\setminus \Sigma_i$, labelled so that $\Omega^-_0\cap \Omega^-_1 = \emptyset$ (in other words $\Sigma_0 \subset \Omega_1^+$ and $\Sigma_1 \subset \Omega_0^+$). 

If we are in Case 2, then we replace $\Sigma_0$ by a perturbation as follows. Let $H_0\subset \mathbb{S}^{n+1}$ be the closed totally geodesic hypersurface containing $\Sigma_0$, and consider the parallel hypersurface $H^\epsilon_0$ of distance $\epsilon$ which intersects $\Omega^+_0$ (this the boundary of a geodesic ball in the sphere). The hypersurface $H^\epsilon_0$ has strictly positive mean curvature (pointing away from $H_0$). For suitably small $\epsilon > 0$, we redefine $\Sigma_0 := H^\epsilon_0 \cap N$. In this case, we also redefine $\Gamma_0 := \partial \Sigma_0 = H^\epsilon_0 \cap \pr N$ (which is now a strictly convex sphere in $\partial N$) and redefine $\Omega_0^{\pm}$ as above. Note that we choose $\epsilon>0$ sufficiently small so that $\Sigma_0$ remains disjoint from $\Sigma_1$. If we are in Case 3, then we leave $\Sigma_0$ unchanged. 

Now $\Omega:= \Omega_0^+ \cap \Omega_1^+$ has boundary which decomposes as $\partial \Omega= \Sigma_0 \sqcup \Sigma_1  \sqcup H' \sqcup \Gamma_0 \sqcup \Gamma_1$, where $H'$ is an open subset of $\partial N$, each $\Sigma_i$ has nonnegative mean curvature with respect to $\Omega$, and $\Gamma_i = \pr \Sigma_i$. Note $\partial H' = \Gamma_0 \sqcup \Gamma_1$ (compare Figure \ref{fig:domain} below). Fix a curve $\sigma$ which connects a point on $\Sigma_0$ with a point on $\Sigma_1$, and whose interior lies in $\Omega$. 

Having described our domain $\Omega$, we next choose a suitable boundary and a homology class for the Plateau problem.  
Consider a hypersurface $\Sigma'$, homologous to $\Sigma_0$, which separates $\Omega$; we may choose $\Sigma'$ to be to be smoothly and properly embedded in $\Omega$ with smooth boundary $\Gamma\subset H'$. Moreover, since $\Gamma_0 \cap \Gamma_1 = \emptyset$, we may arrange\footnote{For instance, consider the set of points in $\Omega$ a fixed distance $\epsilon>0$ from $\Sigma_0$, take a smooth approximation and use the room near the boundary to perturb if needed.} that $\Gamma$ is nowhere totally geodesic.\footnote{We remark that our proof strategy is not sensitive to whether $\Gamma$ is empty or not. However, in this setting, $\Gamma$ is necessarily nonempty as a consequence of the classical Frankel property.}

We find a current $T$ which is area-minimizing among $n$-currents supported in $\overline{\Omega}$ which are homologous to $[[\Sigma']]$ and satisfy $\partial T= [[\Gamma]]$ (c.f. \cite{FH}). Note that the (mod 2) intersection number of $\sigma$ with $[[\Sigma_0]]$ (hence $[[\Sigma']]$) is 1, so the same is true of $\sigma$ with $T$. 

We now rule out touching the barriers $\Sigma_i$. Suppose for the sake of contradiction that $\spt T\cap \Sigma_1 \neq \emptyset$. As $\Sigma_1$ is minimal with boundary in $\partial N$, it is certainly stationary in $(N,\pr N)$. Moreover, White's strong maximum principle (see Appendix \ref{sec:white}) implies that $\Sigma_1 \subset \spt T$; in fact, there is a neighbourhood $U$ of $\Sigma_1$ in $\overline{\Omega}$ so that $\spt T \cap \overline{U} = \Sigma_1$. But $T$ was, in particular, locally minimising away from $\Gamma$, and  $\Gamma$ is disjoint from $\overline{\Sigma_1}$ (hence $\overline{U}$). As any nonnegative function (multiplied by the inward normal) on $\Sigma_1$ will generate a variation of $\Sigma_1$ in $U$, it follows that $\Sigma_1$ must be stable in $(N,\pr N)$ with respect to one-sided variations. However, $\Sigma_1$ was assumed to not be totally geodesic, so this contradicts Proposition \ref{prop:current-bernstein}. As for $\Sigma_0$, either the same argument applies (if we did not perturb it), or else $\Sigma_0$ is strictly mean convex so White's maximum principle already implies $\spt T$ and $\Sigma_0$ are disjoint. 

Thus we have shown that $\spt T \cap \Sigma_i = \emptyset$ for each $i$. Again, since $T$ was (locally) minimising in $\overline{\Omega}$ amongst currents with fixed boundary $\Gamma\subset \pr N$, it follows that the (interior) regular part $\Sigma:= \reg T\subset\Omega$ is stable in $(N,\pr N)$. Moreover, by (interior) regularity for minimisers (e.g. \cite{Mo03}), we have $\dim_\mathcal{H} \sing T < n-2$. Note that $\Sigma$ is nonempty, as $\spt T$ had to intersect $\sigma$. 

But now Proposition \ref{prop:current-bernstein} implies that $\Sigma$ is totally geodesic. By the maximum principle\footnote{This says that if a stationary $n$-varifold has singular set of Hausdorff dimension less than $n-1$, then its support is connected if and only if its regular part is connected. } of Ilmanen \cite{Ilm96} and the Hausdorff estimate for the singular set, each connected component $\Sigma^{(i)}$ of $\spt |T|$ is a totally geodesic half-equator in $\overline{N}$.  It follows that each boundary $\pr \Sigma^{(i)}$ is in fact the same totally geodesic equator $\pr \Sigma$ in $\pr N=\mathbb{S}^n$. 

It remains to analyse the edges $\Gamma_i$. Suppose that $\pr \Sigma \subset \Gamma_i$. Then $\overline{\Sigma} \cap \overline{\Sigma_j}=\emptyset$ for $j\neq i$, and we have already found that $\Sigma$ is totally geodesic. That is, if we were in Case 3, we can repeat the argument on $\Sigma,\Sigma_j$, which reduces to Case 2. 

If we were already in Case 2 ($\Sigma_0$ totally geodesic), then the perturbed boundary $\Gamma_0 = \pr \Sigma'_0$ is certainly nowhere totally geodesic, so $\pr \Sigma$ cannot be contained in $\Gamma_0$. Then $\pr \Sigma \subset \Gamma_1$, and $\overline{\Sigma} \cap \overline{\Sigma}_0=\emptyset$. But in this case both $\Sigma, \Sigma_0$ are totally geodesic, which is impossible by Case 1. 

Thus we may assume that $\pr \Sigma$ is not contained in either $\Gamma_i$. But now since $\Gamma$ was nowhere totally geodesic, $\pr \Sigma$ also cannot be contained in $\Gamma$. But then $\pr\Sigma$, hence $\spt |T|$, must contain a point $p \in (\pr\Omega \cap \pr N) \setminus (\Gamma_0 \cup \Gamma_1\cup \Gamma) =  H' \setminus \Gamma$. White's strong maximum principle then implies that there is a neighbourhood $U'$ of $p$ so that $\spt |T| \cap U' \subset H'$. This contradicts that $p\in\pr\Sigma$, and completes the proof.

\subsection{Proof of Theorem \ref{thm:hemisphere}}

Suppose $\Sigma_0$ and $\Sigma_1$ are closed (compact without boundary) $n$-dimensional, embedded minimal hypersurfaces of the sphere $\mathbb{S}^{n+1}$. For $a \in \mathbb{S}^{n+1}$, let $N_a = \mathbb{S}^{n+1} \cap \{\langle x,a\rangle > 0\}$ be the open hemisphere centred at $a$. Let $\mathcal{I} := \{ a\in \mathbb{S}^{n+1} | \overline{N}_a \cap \overline{\Sigma}_0 \cap \overline{\Sigma}_1 \neq \emptyset\}$. By continuity, the set $\mathcal{I}$ is closed. Now by the transversality theorem, for almost every $a\in \mathbb{S}^{n+1}$, $\pr N_a$ will intersect each $\Sigma_i$ transversely. For such $a$, we have $(\Sigma_i, \Gamma_i) \subset (N_a, \partial N_a)$ is proper, so Theorem \ref{thm:hemisphere-more-general} applies, and yields that $a\in \mathcal{I}$. We conclude that $\mathcal{I} = \mathbb{S}^{n+1}$ as desired. 

\section{Bernstein theorems and Frankel properties in geodesic (half-)balls}
\label{sec:bernstein-fb-full}

In this section, we give a proof of the improved Frankel property for embedded free boundary minimal hypersurfaces in (half-)geodesic balls in a space-forms. 

\subsection{Bernstein result for stable currents in geodesic (half-)balls in space forms}
\label{sec:bernstein-fb-formal}

We now prove the analogue of Proposition \ref{prop:fb-smooth-bernstein} for stable currents. In the following proposition, let $M$ be one of the space forms and $N = B^{n+1}_{R, +}$ be a geodesic half-ball as described in Section 2. Recall the boundary of the half-ball consists of a geodesic closed component $D = \overline{\partial B^{n+1}_{R, +} \setminus \partial B^{n+1}_R}$ and an open component $S = \partial B^{n+1}_{R, +} \setminus D$ lying the boundary of full ball. 

\begin{proposition}\label{prop:fb-current-bernstein}
Suppose that $T \in \mathbb{I}_n(N, \partial N)$ is an integral $n$-current in a geodesic half-ball which is stationary in $(N, D)$. Assume the regular set $\Sigma = \reg T$ has finite area $\mathcal{H}^n(\Sigma) < \infty$, and the (relative) singular set $\sing T$ satisfies the Hausdorff dimension estimate $\mathrm{dim}_{\mathcal{H}} \sing T < n - 2$. If $\Sigma$ is stable in $(N, D)$ with respect to one-sided variations, then its second fundamental form $A_{\Sigma}$ vanishes everywhere.
\end{proposition}

\begin{proof}
The proof is essentially identical to the one given for Proposition \ref{prop:current-bernstein}. Here we let $\rho$ denote the distance to $D \subset \partial N$.  Let us introduce the notation $\partial_S \Sigma := \Sigma \cap S$ for the boundary regular points in $\Sigma$.

We take $u := \sn \rho |_{\Sigma} > 0$ as before, noting now that in addition to $L_{\Sigma} u = |A_{\Sigma} |^2 u$  we also have $(\partial_\eta - A_S(\nu, \nu))u = 0$. Arguing as we did above, we obtain
\begin{align*}
0&\leq -\int_\Sigma uv L_\Sigma(uv) + \int_{ \pr_S \Sigma} uv (\pr_\eta - A_S(\nu,\nu)) (uv)  \\
& = -\int_\Sigma \left( v^2 uL_\Sigma u +  2uv \langle \nabla u, \nabla v\rangle + u^2 \phi \Delta_{\Sigma} v \right) + \int_{ \pr_S \Sigma} \left(v^2 u(\pr_\eta - A_S(\nu,\nu))u + u^2 v \pr_\eta v\right) \\
& =  -\int_\Sigma \left( v^2 |A_\Sigma|^2 u^2 - u^2 |\nabla^\Sigma v|^2  \right) - \int_\Sigma \mathrm{div}_{\Sigma}( u^2 v \nabla^\Sigma v) + \int_{\pr_S \Sigma} u^2 v\pr_\eta v\\
& =  -\int_\Sigma \left( v^2 |A_\Sigma|^2 u^2 -  u^2 |\nabla^{\Sigma} v|^2 \right).
\end{align*}

This implies that for any nonnegative Lipschitz function $v$ with compact support on $\Sigma$, we will have the inequality 
\begin{equation}
\label{eq:current-bernstein}
\int_\Sigma v^2 |A_\Sigma|^2 u^2 \leq  \int_\Sigma u^2 |\nabla^{\Sigma} v|^2, 
\end{equation}
as before. The remainder of the argument goes through unchanged. 
\end{proof}

\subsection{Proof of Theorem \ref{thm:free-boundary-more-general}}

In this section, we establish the improved Frankel property stated in Theorem \ref{thm:free-boundary-more-general}. Again, the proof follows essentially as in the proof of Theorem \ref{thm:hemisphere-more-general}. For the convenience of the reader, we reproduce it below, but with certain modifications to handle the free boundary case: We adopt the convention that the hypersurfaces $\Sigma_i$ to include the portions of their boundary that intersect $S$ (the free boundary portions); this simplifies the notation, but means we need to use the Li-Zhou strong maximum principle to rule out certain touching along the free boundary. 


\begin{remark}
In the proof of Theorem \ref{thm:free-boundary-more-general}, we will produce a solution of the Plateau problem which lies between two free boundary minimal hypersurfaces $\Sigma_0, \Sigma_1$. The Plateau boundary $\Gamma$ will be a perturbation of the intersection of $\Sigma_0$ with the totally geodesic boundary of the half-ball. This intersection would be nonempty if a Frankel property were known for free boundary minimal hypersurfaces in the full ball (this is known when both interior and boundary curvatures are positive \cite{FL14}; for the remaining cases, the nonempty intersection would also follow from the maximum principle and/or our result for minimal hypersurfaces in the hemisphere). In any case, our proof is not sensitive to whether this intersection is empty or not, so it is actually a consequence that it cannot be empty.
\end{remark}

\begin{figure}\label{fig:domain}
\includegraphics[scale=0.18]{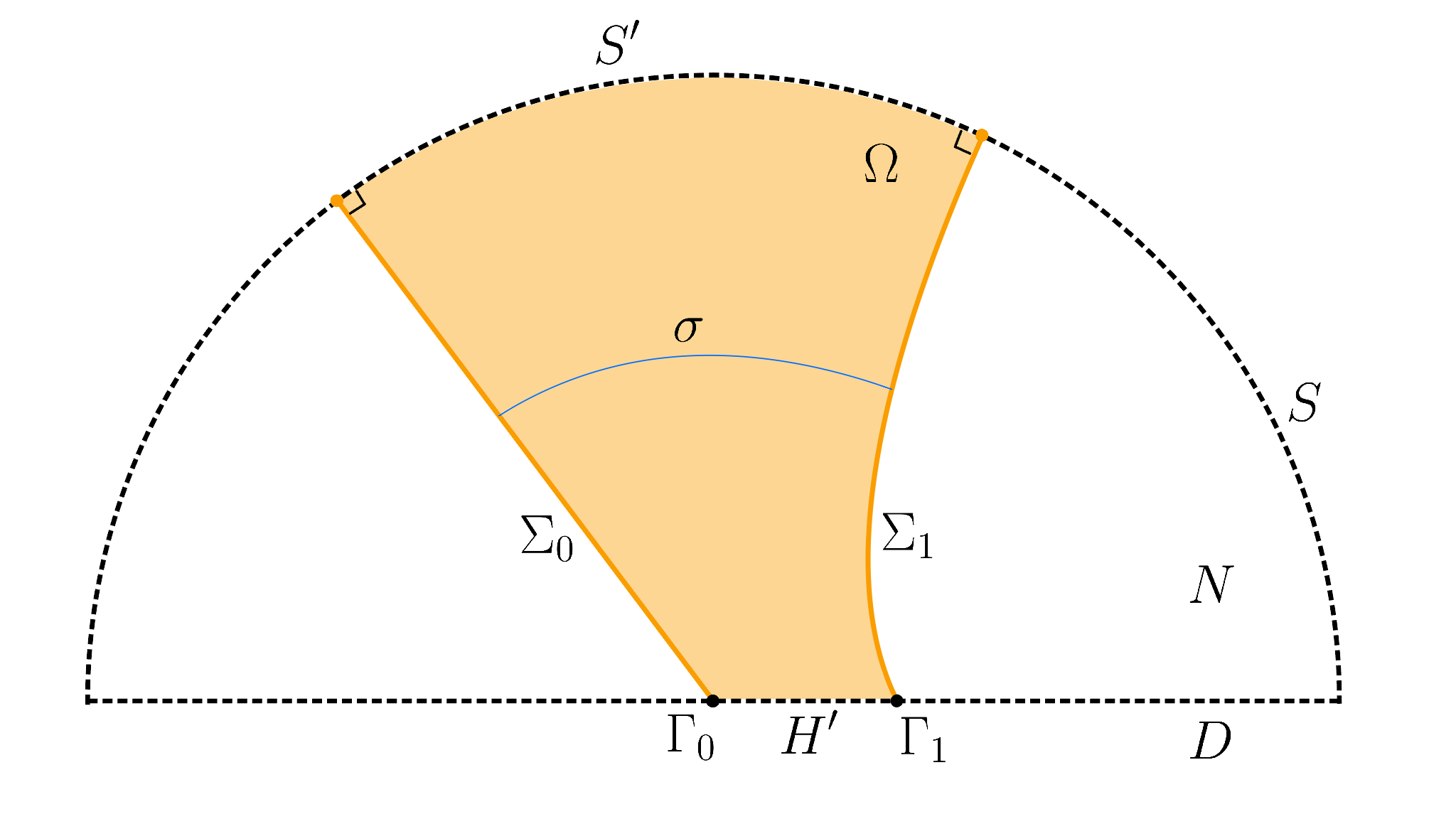}
\caption{Schematic of domain constructed in the proof.}
\end{figure}


Let $N= B^{n+1}_{R,+}$, $D= \overline{\pr B^{n+1}_{R,+} \setminus \pr B^{n+1}_{R}}$, $S= B^{n+1}_{R,+}\setminus D$ and consider connected, properly embedded, $S$ free boundary minimal hypersurfaces $\Sigma_i \subset \overline{N} \setminus D$, $i=0,1$. We emphasize once more our convention that $\partial_S \Sigma := \overline{\Sigma_i} \cap S \subset \Sigma_i$. We again have three cases: 
\begin{itemize}
\item Case 1: Both $\Sigma_i$ are totally geodesic;

\item Case 2: $\Sigma_0$ is totally geodesic, but $\Sigma_1$ is not; or

\item Case 3: Neither $\Sigma_i$ is totally geodesic. 
\end{itemize}

In Case 1, each $\Gamma_i$ must be totally geodesic in $\pr N$, and hence intersect (e.g. by the classification of totally geodesic hypersurfaces). Thus we may assume that we are in Case 2 or Case 3. We suppose for the sake of contradiction that $\overline{\Sigma}_0 \cap \overline{\Sigma}_1 = \emptyset$. 

Our first step is to describe a suitable domain and a suitable boundary for the Plateau problem. The hypersurfaces $\Sigma_0, \Sigma_1$ are both separating in the half-ball $N$. Let $\Omega^{\pm}_i$ be the connected components of $N\setminus \Sigma_i$, labelled so that $\Omega^-_0\cap \Omega^-_1 = \emptyset$ (in other words $\Sigma_0 \subset \Omega_1^+$ and $\Sigma_1 \subset \Omega_0^+$). 

If we are in Case 2, then we replace $\Sigma_0$ by a perturbation as follows. By the classification of totally geodesic hypersurfaces, there is a totally geodesic free boundary hypersurface $H_0\subset B^{n+1}_R$ which contains $\Sigma_0$. By Appendix \ref{sec:conformal-deformations} (see also Remark \ref{rmk:instability}), there is a deformation of $H_0$ by free boundary hypersurfaces $H^\epsilon_0$ in $B^{n+1}_R$, each of which lies to one side of $H_0$ and has strictly positive mean curvature (pointing away from $H_0$). 


For suitably small $\epsilon > 0$, we redefine $\Sigma_0 := H^\epsilon_0 \cap N$. In this case, we also redefine $\Gamma_0 := \partial \Sigma_0 = H^\epsilon_0 \cap \pr N$ and redefine $\Omega_0^{\pm}$ as above. Note that we choose $\epsilon>0$ sufficiently small so that $\Sigma_0$ remains disjoint from $\Sigma_1$. If we are in Case 3, then we leave $\Sigma_0$ unchanged. 

Now $\Omega:= \Omega_0^+ \cap \Omega_1^+$ has boundary which decomposes as $\partial \Omega= \Sigma_0 \sqcup \Sigma_1 \sqcup H'  \sqcup S' \sqcup \Gamma_0 \sqcup \Gamma_1$, where $H'$ is an open subset of $D$, $S'$ is an open subset of $S$, each $\Sigma_i$ has nonnegative mean curvature with respect to $\Omega$, and $\Gamma_i = \pr \Sigma_i \cap D$.  Note $\partial H' = \Gamma_0 \sqcup \Gamma_1$ (see Figure \ref{fig:domain} above). Fix a curve $\sigma$ which connects a point on $\Sigma_0$ with a point on $\Sigma_1$, and whose interior lies in $\Omega$.

Having described our domain $\Omega$, we next choose a suitable boundary and a homology class for the Plateau problem.  
Consider a hypersurface $\Sigma'$, homologous to $\Sigma_0$, which separates $\Omega$; we may choose $\Sigma'$ to be to be smoothly and properly embedded in $\Omega$ with smooth boundary $\Gamma\subset H'$. Moreover, since $\Gamma_0 \cap \Gamma_1 = \emptyset$, we may arrange that $\Gamma$ is nowhere totally geodesic.

We find a current $T$ which is area-minimizing among $n$-currents supported in $\overline{\Omega}$ which are homologous to $[[\Sigma']]$ and satisfy $\spt([[\Gamma]] - T ) \subset S$ (c.f. \cite{FH}, \cite{GR87}). Note that the (mod 2) intersection number of $\sigma$ with $[[\Sigma_0]]$ (hence $[[\Sigma']]$) is 1, so the same is true of $\sigma$ with $T$. 

We now rule out touching the barriers $\Sigma_i$. Suppose for the sake of contradiction that $\spt T\cap \Sigma_1 \neq \emptyset$. As $\Sigma_1$ is minimal with free boundary along $S=\pr N\setminus D$, it is certainly stationary in $(N,D)$. Moreover, the strong maximum principle (which also applies at orthogonal contact points along $S$; see Appendix \ref{sec:white}) implies that $\Sigma_1 \subset \spt T$; in fact, there is a neighbourhood $U$ of $\Sigma_1$ in $\overline{\Omega}$ so that $\spt T \cap \overline{U} = \Sigma_1$. But $T$ was, in particular, locally minimising away from $\Gamma$, and $\Gamma$ is disjoint from $\overline{\Sigma_1}$ (hence $\overline{U}$). As any nonnegative function (multiplied by the inward normal) on $\Sigma_1$ will generate a variation of $\Sigma_1$ in $U$, it follows that $\Sigma_1$ must be stable in $(N,D)$ with respect to one-sided variations. However, $\Sigma_1$ was assumed to not be totally geodesic, so this contradicts Proposition \ref{prop:fb-current-bernstein}.  As for $\Sigma_0$, either the same argument applies (if we did not perturb it), or else $\Sigma_0$ is strictly mean convex so the strong maximum principle already implies $\spt T$ and $\Sigma_0$ are disjoint. 

Thus we have shown that $\spt T \cap \Sigma_i = \emptyset$ for each $i$. Again, since $T$ is (locally) area-minimising amongst currents with fixed boundary $\spt(\pr T) \setminus S = \Gamma\subset D$ it follows that the (interior) regular part $\Sigma:= \reg T\subset\Omega \cup S'$ is stable in $(N,D)$. Moreover, by regularity for minimisers (\cite{GR87}; cf. \cite[Proof of Lemma 5.5]{LZ21b}), we have $\dim_\mathcal{H} \sing T < n-2$. Note that $\Sigma$ is nonempty, as $\spt T$ had to intersect $\sigma$. 

But now Proposition \ref{prop:current-bernstein} implies that $\Sigma$ is totally geodesic. By the maximum principle of Ilmanen \cite{Ilm96} and the Hausdorff estimate for the singular set, each connected component $\Sigma^{(i)}$ of $\spt |T|$ is the intersection of $H_i$ with $N$, where $H_i$ is a totally geodesic free boundary hypersurface in the full ball $B^{n+1}_R$.  It follows that each boundary $\pr \Sigma^{(i)}$ is in fact the same totally geodesic submanifold $\pr \Sigma= H_i \cap D$. 

It remains to analyse the edges $\Gamma_i$. Suppose that $\pr \Sigma \subset \Gamma_i$. Then $\overline{\Sigma} \cap \overline{\Sigma_j}=\emptyset$ for $j\neq i$, and we have already found that $\Sigma$ is totally geodesic. That is, if we were in Case 3, we can repeat the argument on $\Sigma,\Sigma_j$, which reduces to Case 2. 

If we were already in Case 2 ($\Sigma_0$ totally geodesic), then the perturbed boundary $\Gamma_0 = \pr \Sigma'_0$ is certainly nowhere totally geodesic, so $\pr \Sigma$ cannot be contained in $\Gamma_0$. Then $\pr \Sigma \subset \Gamma_1$, and $\overline{\Sigma} \cap \overline{\Sigma}_0=\emptyset$. But in this case both $\Sigma, \Sigma_0$ are totally geodesic, which is impossible by Case 1. 

Thus we may assume that $\pr \Sigma$ is not contained in either $\Gamma_i$. But now since $\Gamma$ was nowhere totally geodesic, $\pr \Sigma$ also cannot be contained in $\Gamma$. But then $\pr\Sigma$, hence $\spt |T|$, must contain a point $p \in (\pr\Omega \cap D) \setminus (\Gamma_0 \cup \Gamma_1\cup \Gamma) =  H' \setminus \Gamma$. The strong maximum principle then implies that there is a neighbourhood $U'$ of $p$ so that $\spt |T| \cap U' \subset H'$. This contradicts that $p\in\pr\Sigma$, and completes the proof.

\subsection{Proof of Theorem \ref{thm:free-boundary}}

Suppose $\Sigma_0$ and $\Sigma_1$ are $n$-dimensional, embedded, free boundary minimal hypersurfaces in $B^{n+1}_R$. For $a \in \mathbb{S}^{n}$, let $B_a = B^{n+1}_R \cap \{\langle x,a\rangle > 0\}$ denote the half-ball in the direction $a$ and $H_a = \{\langle x,a\rangle =0\}$. Also let $D_a = \pr B_a \cap H_a$, $S_a= \pr B_a\setminus H_a$ denote Dirichlet and Neumann boundaries respectively. 

Let $\mathcal{I} := \{ a\in \mathbb{S}^{n+1} | \overline{B}_a \cap \overline{\Sigma}_0 \cap \overline{\Sigma}_1 \neq \emptyset\}$. By continuity, the set $\mathcal{I}$ is closed. Now by the transversality theorem, for almost every $a\in \mathbb{S}^{n+1}$, each $\Sigma_i$ will intersect $H_a$ transversely. For such $a$, each $\Sigma_i \cap B_a$ will be a properly embedded, half-free boundary minimal hypersurface in $B_a$; hence Theorem \ref{thm:free-boundary-more-general} applies, and yields that $a\in \mathcal{I}$. We conclude that $\mathcal{I} = \mathbb{S}^{n}$ as desired. 

\section{Bernstein theorems and Frankel properties in subsets of Gaussian space}
\label{sec:gaussian-main}

In this section, we give a proof of the improved Frankel property in the Euclidean ball of radius $2 \sqrt{n}$ of Gaussian space. 

\subsection{Gaussian minimal hypersurfaces and self-shrinkers}
\label{sec:prelim-gauss}

In setting (III), a minimal hypersurface $\Sigma$ with respect to the Gaussian metric $g= \e^{-|x|^2/2n}g_{\mathbb{R}^{n+1}}$ may equivalently be regarded as a weighted minimal hypersurface in $(\mathbb{R}^{n+1},g_{\mathbb{R}^{n+1}})$ with respect to the weight $w:= \e^{-|x|^2/4}$. Such a hypersurface is often called a \textit{shrinker}, and satisfies the shrinker equation $\vec{H} = - \frac{1}{2}x^\perp$, where the mean curvature is calculated with respect to the Euclidean metric. 

The discussion in Section \ref{sec:stability} provides a description of stability for shrinkers interpreted as minimal hypersurfaces with respect to the Gaussian metric. For the remainder of this section, it will instead be convenient to express the stability properties of $\Sigma$ using the Euclidean metric (and the weight $w$; see also \cite{CM23}). 

Working in the Euclidean metric, it is convenient to denote the drift Laplacian 
\begin{equation}
\mathcal{L}_{\Sigma}v = \Delta_{\Sigma}v - \frac{1}{2} \langle x^\top, \nabla v \rangle
\end{equation}
and redefine a stability operator in terms of Euclidean quantities by
\[
L_{\Sigma} := \mathcal{L}_\Sigma + |A_{\Sigma}|^2 + \frac{1}{2}. 
\]

Up to the (smooth) conformal factor, the stability inequality with respect to $g$ is equivalent to the Gaussian-weighted stability inequality 
\begin{equation}\label{eq:gaussian-stability}
0 \leq  -\int_{\Sigma} v L_\Sigma v\, w,
\end{equation}
and the gradient stability inequality with respect to $g$ is equivalent to the gradient Gaussian-weighted stability inequality
\begin{equation}
\label{eq:gaussian-stability-gradient}
\int_{\Sigma} (|A_{\Sigma}|^2 + \frac{1}{2}) v^2 w \leq \int_{\Sigma} |\nabla^{\Sigma} v|^2 w. 
\end{equation}

We consider certain radial test functions:

\begin{lemma}
Let $\Sigma^n \subset \mathbb{R}^{n+1}$ be a hypersurface satisfying the shrinker equation. Let $r$ be the distance to the origin and $u(r) = 4n-r^2$. Then $u L_{\Sigma}u= \frac{1}{2}r^2(4n-r^2)$. 
\end{lemma}
\begin{proof}
As $\vec{H} = -\frac{1}{2}x^\perp$ and $\Delta_{\Sigma} r = \frac{1}{r}(n - |\nabla^\top r|^2) + \langle \vec{H}, \nabla r \rangle$, we first compute for a general radial test function $u = u(r)$ that 
\[ \begin{split} 
\Lap_\Sigma u - \frac{1}{2} \langle x^\top , \nabla u\rangle &=  u'' |\nabla^\top r|^2 + \frac{u'}{r}(n-|\nabla^\top r|^2) -\frac{1}{2} ru'
\\&= \left(\frac{n}{r}-\frac{r}{2}\right)u' + \left(u'' -\frac{u'}{r}\right) |\nabla^\top r|^2. 
\end{split}\]
Now the choice $u(r) = R^2-r^2$ satisfies $u'' = u'/r$, and one may further compute that 
\[ \left(\frac{n}{r}-\frac{r}{2}\right)u u' + \frac{1}{2}u^2 = \frac{1}{2}(R^2+r^2-4n)(R^2-r^2).\]
Taking $R^2= 4n$ gives the result. 
\end{proof}

This radial test function vanishes on $\pr B^{n+1}_{2\sqrt{n}}$, and in the stability inequality \eqref{eq:gaussian-stability} would give $\int_\Sigma \left( |A_\Sigma|^2 u^2 + \frac{1}{2}r^2(4n-r^2) \right)w \leq 0$. As the second term is strictly positive inside $B^{n+1}_{2\sqrt{n}}$, we have:

\begin{proposition}\label{prop:gaussian-ball-smooth-bernstein}
Suppose that $\Sigma^n \subset B^{n+1}_{2\sqrt{n}}$ is a properly embedded hypersurface satisfying the shrinker equation. Then $\Sigma$ is unstable. 
\end{proposition}

As before, we will need the version of Proposition \ref{prop:gaussian-ball-smooth-bernstein} for singular shrinkers.

\subsection{Bernstein result for stable currents in Gaussian space}
\label{sec:bernstein-gaussian-formal}

We now prove the analogue of Proposition \ref{prop:gaussian-ball-smooth-bernstein} for stable currents.

\begin{proposition}\label{prop:gaussian-current-bernstein}
Let $N=B^{n+1}_{2\sqrt{n}}\subset \mathbb{R}^{n+1}$ denote the ball of Euclidean radius $2\sqrt{n}$. Suppose that $T \in \mathbb{I}_n(N, \partial N)$ is an integral $n$-current which is stationary in $(N, \pr N)$ with respect to the Gaussian metric. Assume the regular set $\Sigma = \reg T$ has finite Gaussian area $\int_{\Sigma} w \,d\mathcal{H}^n < \infty$, and the (relative) singular set $\sing T$ satisfies the Hausdorff dimension estimate $\mathrm{dim}_{\mathcal{H}} \sing T < n - 2$. If $\Sigma$ is stable in $(N, \pr N)$ with respect to one-sided variations, then its Euclidean second fundamental form $A_{\Sigma}$ vanishes everywhere.
\end{proposition}

\begin{proof}[Proof of Proposition \ref{prop:gaussian-current-bernstein}.]
Note in the statement, area (hence stationarity and stability) are understood in the Gaussian metric. 
Working with the Euclidean structure in what follows, take $u$ to be: the (Euclidean) distance to $\pr N$ if $N$ was a half-space; or the radial function $r = 4n-|x|^2$ if $N$ was a ball. In either case, we have $L_\Sigma u \geq |A_\Sigma|^2 u$. As in the previous settings, for any smooth $v\geq 0$ the Gaussian stability inequality \eqref{eq:gaussian-stability} gives
\begin{align*}
0&\leq -\int_\Sigma uv L_\Sigma(uv)\, w   \\
& = -\int_\Sigma \left( v^2 uL_\Sigma u +  2uv \langle \nabla u, \nabla v\rangle + u^2 \phi \mathcal{L}_{\Sigma} v \right) w \\
& \leq  -\int_\Sigma \left( v^2 |A_\Sigma|^2 u^2 - u^2 |\nabla^\Sigma v|^2  \right) w  + \mathrm{div}_{\Sigma}( u^2 v (\nabla^\Sigma v) w) \\
& =  -\int_\Sigma \left( v^2 |A_\Sigma|^2 u^2 -  u^2 |\nabla^{\Sigma} v|^2 \right) w,
\end{align*}
and in particular the final inequality will hold for any Lipschitz $v\geq 0$. 

Let $N_{\delta} = \{x\in N \;: \; d(x,\pr N) \geq \delta\}$, where the distance is again Euclidean. We assert that for any $\delta, \epsilon >0$, there is a Lipschitz cutoff $\phi : N \to [0,1]$ which vanishes in a neighbourhood of $\sing T \cap N_\delta$, and which satisfies $\int_{\Sigma \cap {N_\delta}} |\nabla^\Sigma \phi|^2 w \leq\epsilon$. For instance, one may use the cutoffs constructed in \cite[Section 5]{Z20} (with $R= 2\sqrt{n}-\delta$). 

Now again take $\psi = \psi_{\delta} : N \to [0,1]$ to be a smooth cutoff function on $N$ such that $\psi \equiv 1$ in $N_{2\delta}$, $\psi \equiv 0$ on $N\setminus N_\delta$, and $|\nabla \psi| \leq C/\delta$ in the cutoff region. Again arguing as in the proof of Proposition \ref{prop:current-bernstein}, taking $v = \psi\phi$, we will have 
\[ 
\int_\Sigma |A_\Sigma|^2 u^2 \psi^2 \phi^2 w \leq 2\epsilon + 2\int_\Sigma u^2 \phi^2 |\nabla^\Sigma \psi|^2 w \leq 2\epsilon + C\int_{\Sigma \cap N_{2\delta}\setminus N_\delta} w \, d\mathcal{H}^n. 
\] 

where $C$ is independent of $\delta,\epsilon$. As $\Sigma$ has finite Gaussian area, taking $\delta,\epsilon\to0$ yields that $\int_\Sigma |A_\Sigma|^2 u^2 =0$, hence $A_\Sigma\equiv0$ as desired.
\end{proof}

\subsection{Proofs of Theorem \ref{thm:ball-shrinker-more-general}}

The proofs of Theorem \ref{thm:ball-shrinker-more-general} is essentially the same as the proof of Theorem \ref{thm:hemisphere-more-general} verbatim, except that:

\begin{itemize}
\item Area, stationarity, stability and mean curvature should be understood with respect to the Gaussian metric, except that `totally geodesic' should be understood as totally geodesic in the Euclidean metric;
\item The perturbation in Case 2 should proceed using the parallel hyperplanes $H^\epsilon_0$ of Euclidean distance $\epsilon$ from the (Euclidean) totally geodesic plane $H_0$. Note that these have vanishing Euclidean mean curvature but strictly positive \textit{Gaussian} mean curvature: If $H_0 = \{x_0=0\}$ so that $H_0^\epsilon = \{x_0 =\epsilon\}$, then the Gaussian mean curvature of $H^\epsilon_0$ is given by \[\e^{|x|^2/4n}\frac{1}{2} \langle x, \nu\rangle = \e^{|x|^2/4n}\frac{1}{2} \langle x, e_0\rangle = \frac{\epsilon}{2} \e^{|x|^2/4n}.\] 
The perturbed surface is then $H^\epsilon_0 \cap B^{n+1}_{2\sqrt{n}}$.
\end{itemize}

\section{A variational approach to half-space type Frankel properties}

In the following sections, we establish differential inequalities for the distance to a hypersurface. Such differential inequalities have been established before in various settings, and are typically proved either via a length variation argument or via Riccati comparison and the Bochner formula - cf. Petersen-Wilhelm \cite{PW03}, Choe-Fraser \cite{CF18}, Wei-Wylie \cite{WW09} and Moore-Woolgar \cite{MW21}. Riccati comparison is rather efficient where the distance function is smooth, while the length variation approach appears to be more effective at non-smooth points of the distance function. The comparison result we need is closest to the one established in \cite{MW21}, but they do not appear to address the length variation proof for the case we need, which is needed to establish the differential inequality in full strength. These differential inequalities may be used to prove certain Frankel properties for \textit{weighted} minimal hypersurfaces when the ambient $\alpha$-Bakry-\'{E}mery-Ricci curvature is nonnegative (cf. \cite{MW21}, and Proposition \ref{prop:d-maximum-principle} below).

These results may be used to give an alternative proof of the half-space Frankel properties. In each of the setting of interest (i.e. settings (I), (II), and (III) from Section \ref{sec:prelims}), the appropriate half-space becomes isometric to hyperbolic space after a conformal change by the same test functions we used in the stability arguments above. Moreover, we find that the $\alpha$-Bakry-\'{E}mery-Ricci curvature is nonnegative for a suitable $\alpha$. (In fact, we will need the critical \textit{negative} parameter value $\alpha=-n$.) Considering the distance functions with respect to the conformally changed metric will then allow us to deduce the half-space Frankel properties. 

\subsubsection{Weighted notions}

Recall that for $\alpha \in (\mathbb{R}\cup\{\infty\}) \setminus \{0\}$ and $f$ a smooth function, the $\alpha$-Bakry-\'{E}mery-Ricci tensor is given by $\Ric^\alpha_f := \Ric + \nabla^2 f - \frac{1}{\alpha} df\otimes df$ and the drift Laplacian is given by $\Lap_f := \Lap - \langle \nabla f, \nabla \cdot\rangle$. The $f$-mean curvature vector of a hypersurface $\Sigma$ is given by $\vec{H}_{\Sigma, f} := \vec{H}_{\Sigma} + \nabla^\perp f$. A hypersurface $\Sigma$ is said to be $f$-minimal in $(M, g, e^{-f})$ if $\vec{H}_{\Sigma, f} = 0$. 

\subsection{Conformal change in space forms}

Recalling the notation introduced previously in Section \ref{sec:test-function}, we first prove the following lemma.

\begin{lemma}
\label{lem:conformal-change}
Let $(M , g)$ be an $(n+1)$-dimensional simply connected space form with curvature $\kappa \in \{-1, 0, 1\}$. Consider any complete totally geodesic hypersurface $M' \subset M$ and let $\tilde{M}$ be one of the connected components of $M \setminus M'$. Let $\rho: M \to [0, \infty)$ be the distance function to $M'$. Let $e^{-f} = (\sn \rho)^n$. 

Then:
\begin{enumerate}
\item $(\tilde{M}, \tilde{g})$ is complete with respect to the conformal metric $\tilde{g} = (\sn \rho)^{-2} g = e^{\frac{2}{n} f} g$. In fact, $(\tilde{M}, \tilde{g})$ is isometric to $\mathbb{H}^{n+1}$. 
\item Minimal hypersurfaces in $(\tilde{M}, g)$ correspond to $f$-minimal hypersurfaces in $(\tilde{M}, \tilde{g}, e^{-f})$. 
\item $(\tilde{M}, \tilde{g}, e^{-f})$ has $\wt{\Ric}^\alpha_f = 0$, where $\alpha =-n$. 
\item If $B^{n+1}_R$ is any geodesic ball in $(M, g)$ of radius $R$, centred at a point on $M'$, then $\pr B^{n+1}_R \cap \tilde{M}$ is totally geodesic in $(\tilde{M},\tilde{g})$.
\end{enumerate} 
\end{lemma}
\begin{proof}

Let $\phi = -\log \sn \rho = \frac{1}{n} f$, so that $\tilde{g}= e^{2\phi} g$. Then $d\phi = - \frac{1}{\sn \rho} d(\sn \rho)$, so $|d\phi|_{g}^2 = \ct^2\rho$ and (recalling Lemma \ref{lem:space-form-test-function}) 
\[
\nabla^2\phi = -\frac{1}{\sn \rho} \nabla^2 \sn\rho + \frac{1}{\sn^2\rho} d(\sn\rho)\otimes d(\sn\rho) = \kappa g+ d\phi \otimes d\phi. 
\]
In the notation of Appendix \ref{sec:conformal-change}, we then have $\Phi =(\kappa+ \frac{1}{2}\ct^2\rho)g$, and hence (recalling $\ct^2 \rho + \kappa  =(\sn \rho)^{-2}$)  the Riemann curvature of conformal metric $\tilde{g}$ is 
\[
\tilde{R}_{ijkl} = -\frac{\kappa + \ct^2\rho}{\sn^2\rho}(g_{ik}g_{jl} - g_{il}g_{jk}) = -(\tilde{g}_{ik}\tilde{g}_{jl} - \tilde{g}_{il} \tilde{g}_{jk}).
\] 
Thus $(\tilde{M},\tilde{g})$ has constant curvature $-1$. As $\tilde{M}$ is also simply-connected this establishes (1). 

For (2), if $\Sigma \subset \tilde{M}$ is a hypersurface, then the weighted conformal volume (with $f = n \phi$) is precisely related to $g$-induced volume form by $e^{-f} d\tilde{\mu}_{\Sigma} = e^{-f} e^{n\phi} d\mu_{\Sigma} =  d\mu_{\Sigma}$. Thus $f$-minimal hypersurfaces with respect to $\tilde{g}$ are exactly minimal with respect to $g$. 

For (3), we already have $\wt{\mathrm{Ric}} = -n \tilde{g}$ and the formula for the Hessian under conformal change Appendix \ref{sec:conformal-change} gives
\begin{align*}
\wt{\nabla}^2 f &= n(\nabla^2 \phi- 2 d\phi \otimes d\phi + |d\phi|^2 g) \\
& = n (\kappa + \ct^2 \rho) g - n \, d\phi \otimes d\phi \\
& = n \tilde{g}  - n \, d \phi \otimes d\phi. 
\end{align*}
Hence the $\infty$-BE Ricci curvature of $(\tilde{M},\tilde{g},e^{-f})$ is
\[
\wt{\Ric}_f = \wt{\mathrm{Ric}} + \wt{\nabla}^2 f = -n \, d\phi\otimes d\phi = -\frac{1}{n} df\otimes df. 
\]
Consequently, with $\alpha = -n$, we have shown $\wt{\Ric}^\alpha_f = 0$. 

Finally, for (4) recall that the second fundamental form of $S = \partial B^{n+1}_R$ (oriented by the inward-pointing normal $\nu = -\partial_r$ and with respect to $g$) is given by $A_S = (\ct R) g$. Under the conformal change, one obtains
\[
\tilde{A}_S = e^{-\phi}(A_S - \partial_\nu \phi \,g) = 0
\] 
since (as in the proof of Lemma \ref{lem:space-form-test-function}) $\partial_\nu \phi = \ct \rho \langle \nabla r, \nabla \rho \rangle = \ct r$. 

\end{proof}

\subsection{Conformal change in Gaussian space} \label{sec:gaussian-conformal}

If we consider the Euclidean space with the Gaussian weight, $(\mathbb{R}^{n+1}, g_{\mathbb{R}^{n+1}}, e^{-\gamma})$ where $\gamma = r^2/4$ and $r = |x|$, a similar result can be deduced in the Gaussian setting. Namely, given $R > 0$ consider $\rho := R^2 - r^2$; let $\tilde{M} := B^{n+1}_R$ be a Euclidean ball of radius $R$; and take $\tilde{g} = \rho^{-2} g= e^{\frac{2f}{n}}g$ with $e^{-f} = \rho^n$. Then, similarly to the above, $\gamma$-minimal hypersurfaces in $(\tilde{M}, g, \gamma)$ (which are just shrinkers) correspond to $(\gamma + f)$-minimal hypersurfaces in $(\tilde{M}, \tilde{g}, e^{-f-\gamma})$. By straightforward computation, we have 
\[
\tilde{R}_{ijkl} = - \frac{4 R^2}{(R^2 - r^2)^4}(g_{ik} g_{jl} - g_{il} g_{jk}) = - 4R^2 (\tilde{g}_{ik}\tilde{g}_{jl} - \tilde{g}_{il} \tilde{g}_{jk}).
\]
So (as expected) we recognize $(\tilde{M}, \tilde{g})$ as the Poincar\'e disk up to scaling. Now observe 
\begin{align*}
\wt{\nabla}^2 \gamma &= \nabla^2 \gamma - \frac{1}{n} d\gamma \otimes df -\frac{1}{n} df \otimes d\gamma + \frac{1}{n} g(\nabla \gamma, \nabla f) g \\
&= \frac{R^2 + r^2}{2(R^2 -r^2)} g - \frac{1}{n} d\gamma \otimes df -\frac{1}{n} df \otimes d\gamma.
\end{align*}
Computing the weighted Ricci, we find 
\begin{align*}
\wt{\mathrm{Ric}}_{f+\gamma} &= \wt{\mathrm{Ric}} + \wt{\nabla}^2 f + \wt {\nabla}^2 \gamma \\
&= -4n R^2 \tilde{g} +\Big(\frac{2n}{R^2 -r^2} + \frac{4nr^2}{(R^2 - r^2)^2} + \frac{R^2 + r^2}{2(R^2 -r^2)}\Big)g \\
& \qquad - \frac{1}{n} df \otimes df - \frac{1}{n} d\gamma \otimes df -\frac{1}{n} df \otimes d\gamma \\
& =\frac{1}{2}(R^2 - r^2)(R^2 +r^2 - 4n)\tilde{g} - \frac{1}{n} d(f +\gamma) \otimes d(f + \gamma) + \frac{1}{n} d\gamma \otimes d\gamma.
\end{align*}
In particular, taking $R^2 \geq 4n$, we deduce that $(\tilde{M}, \tilde{g}, e^{-f-\gamma})$ with $\alpha = -n$ has 
\[
\wt{\mathrm{Ric}}^\alpha_{f + \gamma} =\frac{1}{2}(R^2 - r^2)(R^2 +r^2 - 4n)\tilde{g}+ \frac{1}{n} d\gamma^2 \geq 0.
\]

\begin{remark}
The same holds for a Gaussian half-space: If we take $\tilde{M} = \mathbb{R}^{n+1}_+$ and $\rho$ to be the Euclidean distance from $\partial \tilde{M}$, then the computations above give that $(\tilde{M}, \tilde{g})$ is isometric to Hyperbolic space with $\wt{\Ric}^\alpha_{f+\gamma} = \frac{1}{n} d\gamma^2 \geq 0$. 
\end{remark}

\subsection{Differential inequalities for the distance to a hypersurface}

\begin{definition}
Given a domain $\Omega \subset M$ with smooth (open) boundary component $S' \subset \partial \Omega$, consider the partial Neumann problem
\begin{equation}\label{eq:bvp}\begin{cases} 
\Lap_f u = 0 &  \text{ in } \Omega \\ 
\langle \nabla u, \eta\rangle =0& \text{ on } S' \subset \partial \Omega,
 \end{cases} \end{equation}
 where $\eta$ is the inward-pointing unit normal on $S' \subset \pr \Omega$. 

Given a continuous function $u$ on $\overline{\Omega}$,
\begin{enumerate}
\item[(i)] at $x \in \overline{\Omega}$, we say $u$ satisfies $\Delta_f u(x) \leq 0$ \textit{in the viscosity sense} if, for any smooth function $\phi$ defined in a neighborhood $V$ of $x$ satisfying $\inf_{V \cap \overline{\Omega}}(u -\phi) = u(x)- \phi(x) = 0$, we have $\Delta_f \phi(x) \leq 0$;
\item[(ii)] at $x \in S'\subset \partial \Omega$, we say  $u$ satisfies $\langle \nabla u, \eta\rangle \leq 0$ \textit{in the viscosity sense} if for any smooth function $\phi$ defined in a neighborhood $V$ of $x$ satisfying $\inf_{V \cap \overline{\Omega}}(u -\phi) = u(x) - \phi(x) = 0$, we have $\langle \nabla \phi(x), \eta(x) \rangle \leq 0$. 
\end{enumerate}
If both (i) and (ii) hold for all $x \in \Omega \cap S'$, we say $u$ is a \textit{viscosity supersolution} of the partial Neumann problem \eqref{eq:bvp}.
\end{definition} 
 
It is known that viscosity supersolutions with Neumann boundary data have a strong maximum principle that holds up to the boundary: If $u$ is a viscosity supersolution as above and $\inf_{\overline{\Omega}} u = u(x)$ is attained at some $x \in \Omega \cup S'$, then $u$ is constant. The setting we consider is quite simple (i.e. linear), but we refer the reader to Theorems 1 and 2 in \cite{KK98}, or Remark 3.2 in the earlier work by Trudinger \cite{Tru88} for quite general results. 

The main result of this section is the following differential inequality for the distance function to a hypersurface. 

\begin{theorem}\label{thm:d-super}
Let $(M, g, e^{-f})$ be a complete $(n+1)$-dimensional manifold satisfying $\Ric^\alpha_f \geq 0 $ where $1 + \frac{n}{\alpha} \geq 0$. Suppose $N$ is a smooth geodesically convex domain in $M$ with (possibly empty) smooth boundary. 

Let $\Sigma$ be a complete, properly embedded, 2-sided, $f$-minimal hypersurface with free boundary in $(N, \pr N)$, and consider the distance function $d_\Sigma: N \to \mathbb{R}$ from $\Sigma$. Then $d_\Sigma$ is a viscosity supersolution of (18) with $\Omega = N \setminus \Sigma$ and $S' = \partial N \setminus \Sigma$.
\end{theorem}

\begin{proof}
The proof is an application of Synge's first and second variation formulae. Given a curve $\gamma:[0,1]\to M$, in what follows let $E(\gamma) := \int_0^1 |\gamma'|^2\,dt$ denote the energy and $L(\gamma) = \int_0^1 |\gamma'| \, dt$ denote the length. 

We first consider the case of interior points. Let $q \in N \setminus \Sigma$ be an interior point and suppose that $\phi$ is a $C^2$ function defined in a neighborhood of $q$ which touches $d_{\Sigma}$ from below. We must show that $\Delta_f \phi(q) \leq 0$. To that end, let $p \in \overline{\Sigma}$ be a point such that $d(p, q) = d_{\Sigma}(q)$. 

Let $\gamma : [0, 1] \to M$ minimizing geodesic from $p$ to $q$.  By assumption, the geodesic $\gamma$ remains in the interior of $N$.  Choose an orthonormal basis of vectors $e_0, \dots, e_n \in T_pM$ such that $e_0$ is parallel to $\gamma'(0)$. Because $\Sigma$ contacts $\partial N$ orthogonally, distance minimization implies $\gamma'(0) \in T_p^\perp \Sigma$ and hence $e_1, \dots, e_n \in T_p \Sigma$. Indeed, this is a well-known consequence of the first variation formula when $p \in \Sigma \setminus \partial \Sigma$ is interior. If $p \in \partial \Sigma \subset \partial N$ lies on the boundary, then the same variational argument implies $\gamma'(0)$ must be orthogonal to the space $T_p \Sigma \cap T_p \partial N$ and satisfy $\langle \gamma'(0), \eta_{\partial \Sigma}(p)\rangle \leq 0$ where $\eta_{\partial \Sigma}$ is the inward-pointing conormal to $\partial \Sigma$ in $\Sigma$. On the other hand, by geodesic convexity $\gamma'(0)$  must point into $N$ so $\langle \gamma'(0), \eta(p) \rangle \geq 0$. But the orthogonality condition gives precisely $\eta_{\partial \Sigma} = \eta$, so $\gamma'(0) \perp \eta_{\partial \Sigma}(p)$ and hence $\gamma'(0) \in T^\perp_p\Sigma$ in this case too. 

For simplicity, let $e_i = e_i(t)$ also denote the parallel extension of $e_i$ along $\gamma(t)$. For any smooth variation of $\gamma_s$ of $\gamma$, recall the first and second variations of energy are given by 
\begin{align*}
\frac{d}{ds} E(\gamma_s) &= \langle V, \gamma' \rangle \big|_{t = 0}^{t = 1},&
\frac{d^2}{ds^2} E(\gamma_s) &= \langle X, \gamma'\rangle \Big|_{t = 0}^{t = 1} + \int_0^1 |\nabla_{\gamma'} V |^2 - R(V, \gamma', V, \gamma') \, dt,
\end{align*}
where $V(t) := \partial_s \gamma_s(t) \big|_{s = 0}$ and $X(t) := \nabla_{(\partial_s \gamma_s)} (\partial_s \gamma_s) (t) \big|_{ s= 0}$ are the velocity and acceleration of the variation. Additionally, if the variation satisfies $s \mapsto \gamma_s(0) \in \Sigma$, then one has the chain of inequalities $E(\gamma_s) - \frac{1}{2} \phi(\gamma_s(1))^2 \geq \frac{1}{2} L(\gamma_s)^2 - \frac{1}{2} d_{\Sigma}(\gamma_s(1))^2 \geq 0$ with equality when $s = 0$. We now make suitable choices of $\gamma_s$ and apply the first and second derivative tests. 

First, given a smooth function $J : [0, 1] \to \mathbb{R}$ satisfying $J(1) = 1$, we can construct variations $\gamma^i_s(t)$ for $i=1, \dots, n$ that satisfy\footnote{For instance, using the inverse function theorem, we could perturb the variation $\tilde{\gamma}^i_s(t) = \exp_{\gamma(t)}(s J(t) e_i(t))$ near $t = 0$ to satisfy the first condition without affecting the other two. }
\begin{itemize}
\item $\gamma^i_s(0) \in \Sigma$, 
\item $V^i(t) := \partial_s \gamma_s^i(t) \big|_{s = 0} = J(t) e_i(t)$, 
\item $X^i(1) := \nabla_{(\partial_s \gamma^i_s)} (\partial_s \gamma^i_s) (1) \big|_{ s= 0} = 0$.
\end{itemize}
The first and second derivative tests applied to inequality $E(\gamma^i_s) \geq \frac{1}{2} \phi(\gamma^i_s(1))^2$ at $s = 0$ yield $\phi(q) \langle \nabla \phi(q), e_i \rangle = 0$ and, noting that $|\gamma'| = d_{\Sigma}(q) = \phi(q)$, 
\begin{align*}
\phi(q) (\nabla^2 \phi)_q(e_i, e_i) &\leq \langle X^i, \gamma' \rangle \Big|_{t=0}^{t=1} + \int_0^1 |\nabla_{\gamma'} V^i |^2 - R(V^i, \gamma', V^i, \gamma') \, dt \\
& = -J(0)^2 \langle A_{\Sigma, p} (e_i, e_i), \gamma'(0) \rangle + \int_0^1 (J')^2 - J^2 R(\gamma', e_i, \gamma', e_i) \, dt. 
\end{align*}
On the other hand, by considering the variation $\gamma^0_s(t) = \exp_{\gamma(t)}(s t\, e_0(t))$, we analogously deduce that $\nabla \phi(q) = \gamma'(1)$ and 
\[
\phi(q) (\nabla^2 \phi)_q(e_0, e_0)  \leq 0. 
\]
Summing over $i$, we obtain 
\[
\phi(q) \Delta \phi(q) \leq - J(0)^2 \langle \vec{H}_{\Sigma}(p), \gamma'(0) \rangle  + \int_0^1 n (J')^2 - J^2 \mathrm{Ric}(\gamma', \gamma') \, dt. 
\]

Next, observe that 
\begin{align*}
- \langle \nabla \phi(q), \nabla f(q) \rangle + J(0)^2\langle \nabla f(p), \gamma'(0) \rangle &= -\int_0^1 \frac{d}{dt}\big(J^2 \langle \nabla f \circ \gamma ,\gamma' \rangle \big)\, dt \\
& = - \int_0^1 2J J' \langle \nabla f \circ \gamma, \gamma' \rangle +  J^2 (\nabla^2 f)(\gamma', \gamma') \, dt.
\end{align*}
Adding this to inequality above, and using the assumptions $\vec{H}_{\Sigma, f}(p) := \vec{H}_{\Sigma}(p)+  \nabla^\perp f(p) = 0$, we obtain 
\begin{align*}
\phi(q) \Delta_f \phi(q) &\leq \int_0^1 n (J')^2 - 2J J' \langle \nabla f \circ \gamma, \gamma' \rangle -  J^2 \mathrm{Ric}_f(\gamma', \gamma')  \, dt \\
 &= \int_0^1 n (J')^2 - 2J J' \langle \nabla f \circ \gamma, \gamma' \rangle - J^2 \frac{1}{\alpha}\langle \nabla f \circ \gamma, \gamma' \rangle^2 \, dt  - \int_0^1 J^2 \mathrm{Ric}_f^\alpha(\gamma', \gamma') \, dt. 
\end{align*}

Finally, take $J(t) := \exp\big(\frac{1}{n}(f(\gamma(t)) - f(q))\big)$ so that $J' = \frac{1}{n} \langle \nabla f \circ \gamma, \gamma' \rangle J$. Note that $J(1) = 1$. Using that $1 + \frac{n}{\alpha} \geq 0$ and $\mathrm{Ric}_f^\alpha \geq 0$, we conclude
\begin{equation}\label{eq:2pt-conclusion}
\phi(q) \Delta_f \phi(q) \leq -n\int_0^1 \Big(1 + \frac{n}{\alpha}\Big) (J')^2\, dt - \int_0^1 J^2 \mathrm{Ric}^\alpha_f(\gamma', \gamma') \, dt \leq 0.
\end{equation}
Since $q \in N \setminus \Sigma$ and $\phi$ were arbitrary, this completes the proof. 

Now we consider boundary points.
Let $q \in \pr N \setminus\Sigma$ be a boundary point and suppose that $\phi$ is a $C^2$ function defined in a neighborhood of $q$ which touches $d_{\Sigma}$ from below. Again, let $p \in \overline{\Sigma}$ be a point such that $d(p, q) = d_{\Sigma}(q)$. Then $\phi(x) \leq d_\Sigma(x) \leq d(p,x)$ for $x$ near $q$ with equality when $x = q$, hence $\langle \nabla \phi(x), \eta(x) \rangle \leq \langle \nabla d_p(x), \eta\rangle \leq 0$ by convexity (again, $\eta$ is the inward normal). 
\end{proof}

\begin{remark}
In the setting of the previous theorem, the inequality $\Delta_f d_{\Sigma}\leq 0$ at a \textit{smooth} point of $d_{\Sigma}$ can also be proved using the drift Bochner formula. Recall the drift Bochner formula is
\begin{equation}\label{eq:drift-bochner} 
\frac{1}{2}\Lap_f |\nabla u|^2 = |\nabla^2 u|^2 + \Ric_f(\nabla u,\nabla u) + g(\nabla \Lap_f u, \nabla u).
\end{equation}
If $\gamma(t)$ is an integral curve of $\nabla d_{\Sigma}$ (the free boundary assumption and geodesic convexity ensure that this curve remains in $N$), then $m_f(t) := \Delta_f d_{\Sigma}(\gamma(t))$ is the $f$-mean curvature of the level set $\{d_{\Sigma} = t\}$. Taking $u = d_{\Sigma}$ above and using Cauchy-Schwarz, one can show 
\[
m_f'(t) \leq  - \frac{1}{n}\Big(1+ \frac{n}{\alpha}\Big) m_f(t)^2 -\frac{2}{n} f'(t) m_f(t)- \mathrm{Ric}^\alpha_f(\gamma'(t), \gamma'(t)),
\]
when $1 + \frac{n}{\alpha} \geq 0$ (here $f'(t) = \frac{d}{dt}f(\gamma(t))$). In particular, when $\mathrm{Ric}^\alpha_f \geq 0$ and $m_f(0) = 0$ (using the integrating factor $\e^{-2f(t)/n}$) the differential inequality yields $m_f(t) \leq 0$ as well. 
\end{remark}

\begin{proposition}\label{prop:d-maximum-principle}

Let $(M, g, e^{-f})$ be a complete $(n+1)$-dimensional manifold satisfying $\Ric^\alpha_f \geq \kappa g$ for $1 + \frac{n}{\alpha} \geq 0$ and $\kappa \geq 0$. Suppose $N$ is a smooth geodesically convex domain in $M$ with (possibly empty) smooth boundary.  

Suppose $\Sigma_0,\Sigma_1$ are complete, properly embedded, 2-sided, $f$-minimal hypersurfaces with free boundary in $(N, \pr N)$. Suppose that $\Sigma_0,\Sigma_1$ are disjoint and let $\Omega$ denote the region in-between. Consider the distance functions $d_i = d(\cdot,\Sigma_i)$ on $\Omega$. 

If $\kappa>0$, then $d_0+d_1$ cannot achieve a minimum on $\overline{\Omega}$.

If $\kappa=0$ and $1 +\frac{n}{\alpha}>0$, and $d_0 +d_1$ does achieve a minimum on $\overline{\Omega}$, then $u=d_0+d_1$ is constant, the $\Sigma_i$ are totally geodesic and $\Omega$ is isometric to $\Sigma_i \times [0,d]$.

If $\kappa=0$ and $1+\frac{n}{\alpha} =0$, and $d_0+d_1$ achieves a minimum on $\overline{\Omega}$, then $u=d_0+d_1$ is constant and the $\Sigma_i$ are totally geodesic with respect to the conformal metric $e^{-\frac{2}{n}f}g$. 
\end{proposition}

\begin{proof}
Suppose $u = d_0+d_1$ achieves a minimum on $\overline{\Omega}$. If the minimum is achieved on either $\Sigma_i$, then it is also achieved on a minimising geodesic connecting the $\Sigma_i$. In any case, there is an interior minimum. By Theorem \ref{thm:d-super}, $u$ is a (viscosity) supersolution of \eqref{eq:bvp}, so the strong maximum principle implies that $u$ is constant. But each $d_i$ was itself a (viscosity) supersolution of \eqref{eq:bvp}, so it follows that
\begin{equation}
\label{eq:f-harmonic}
\Lap_f d_i =0;
\end{equation} in particular each $d_i$ is a smooth function on $\Omega$, and \eqref{eq:f-harmonic} holds in the \textit{classical} sense. 

If $\kappa>0$, then (\ref{eq:2pt-conclusion}) would give $\Lap_f d_i < 0$, which is already a contradiction. So in this case no minimum is possible and we henceforth assume $\kappa = 0$. 

If $1+ \frac{n}{\alpha} >0$ and a minimum is attained, then (\ref{eq:2pt-conclusion}) implies that for every $x \in \Omega$ either $\Lap_f d_i (x) < 0$ or else $g(\nabla f,\nabla d_i) \equiv 0$ along the minimising geodesic connecting $x$ to $\Sigma_i$. Since $\Delta_f d_i = 0$, we must be in the latter case, and so we conclude $g(\nabla f, \nabla d_i) \equiv 0$ everywhere. On the other hand (as $d_i$ is smooth) the drift Bochner formula (\ref{eq:drift-bochner}) applies to give
\begin{equation}\label{eq:drift-bochner-degen} 
|\nabla^2 d_i|^2 = -\Ric_f(\nabla d_i, \nabla d_i) \leq  -\frac{1}{\alpha}g(\nabla f, \nabla d_i)^2.
\end{equation}
In particular, if $1+ \frac{n}{\alpha} >0$ and a minimum is attained, then $\nabla^2 d_i=0$. Hence the $\Sigma_i$ are totally geodesic, and $\Omega$ is isometric to $\Sigma_i \times [0,d]$. 

For the only remaining case, we have $\kappa=0$ and $\alpha =-n$. Continuing with assumption that that a minimum is obtained, we take the limit of (\ref{eq:drift-bochner-degen}) as $d_i\to 0$ to find that $|A_{\Sigma_i}|^2 \leq \frac{1}{n} g(\nabla f,\nu)^2$. On the other hand, by Cauchy-Schwarz we have $|A_{\Sigma_i}|^2 \geq \frac{1}{n} H_{\Sigma_i}^2 = \frac{1}{n} g(\nabla f,\nu)^2$. Thus 
\[
|A_{\Sigma_i}|^2 = \frac{1}{n} g(\nabla f,\nu)^2.
\]
This implies that each $\Sigma_i$ is totally geodesic in the conformal metric $\tilde{g} = e^{-\frac{2}{n} f} g$. Indeed, following Appendix \ref{sec:conformal-change}, the transformation law for the second fundamental form gives 
\[
|\wt{A}_{\Sigma_i}|_{\tilde{g}}^2 = e^{\frac{2}{n}f} \big(|A_{\Sigma_i}|_g^2 + \frac{2}{n} g(\nabla f, \nu) H_{\Sigma_i} + \frac{1}{n} g(\nabla f, \nu)^2\big) = e^{\frac{2}{n}f} \big(|A_{\Sigma_i}|_g^2- \frac{1}{n} g(\nabla f, \nu)^2+ H_{\Sigma_i, f}\big) = 0.
\]

\end{proof}

\subsection{Length variation proof of the Frankel properties in space forms}

We end by giving the alternative proofs of the half-space Frankel property in settings (I) and (II) from Section \ref{sec:prelims}, using the results of the previous section. 

\begin{proof}[Alternative proof of Theorems \ref{thm:hemisphere-more-general} and \ref{thm:free-boundary-more-general}]

Let $(M, g)$ be one of the $(n+1)$-dimensional simply connected space forms of constant curvature $-1, 0,$ or $1$. Let $M' \subset M$ be any complete totally geodesic hypersurface and let $\tilde{M}$ be one of the connected components of $M \setminus M'$ as above. When $R \in (0, \mathrm{diam}(M))$, define $N = B^{n+1}_{R, +}:= B^{n+1}_R \cap \tilde{M}$, where $B^{n+1}_R$ has centre in $M'$. When $M = \mathbb{S}^{n+1}$ and $R \geq \pi$, set $N = \mathbb{S}^{n+1}_+$. Define $D = \overline{N} \cap M'$ and $S = \partial N \setminus D$. 

In any of these settings, suppose $\Sigma_i \subset \overline{N} \setminus D$ for $i = 0, 1$ are connected, properly embedded, $S$ free boundary minimal hypersurfaces. Recall our convention that $\partial_S \Sigma_i \subset \Sigma_i$. Suppose for sake of contradiction that $\overline{\Sigma}_0 \cap \overline{\Sigma}_1 = \emptyset$. In particular, $d(\Sigma_0, \Sigma_1) > 0$. Let $\Omega$ denote the open set in $N$ between $\Sigma_0$ and $\Sigma_1$ and let us fix a reference point $o' \in \Omega$. 

To obtain a contradiction, we consider the conformal metric on the half-space $\tilde{M}$. As before, let $\rho : \tilde{M} \to (0, \infty)$ denote the distance function to $M'$ and define a metric $\tilde{g} := (\sn\rho)^{-2}g$ and a weight $e^{-f} := (\sn\rho)^n$ on $\tilde{M}$. Let $\tilde{d}$ denote the distance function on $\tilde{M}$ with respect to the conformal metric $\tilde{g}$. By Lemma \ref{lem:conformal-change}, $(\tilde{M}, \tilde{g})$ is isometric to hyperbolic space, and the boundary of $N$ in $\tilde{M}$ (which now consists only of $S$) is totally geodesic with respect to $\tilde{g}$. In particular, the domain $N$ is geodesically convex. The hypersurfaces $\Sigma_0, \Sigma_1$ are complete and still meet the boundary $S$ orthogonally. Again, by choice of weight, the $\Sigma_i$ are $f$-minimal hypersurfaces in $(\tilde{M}, \tilde{g}, e^{-f})$.

Consider the function $\tilde{u} := \tilde{d}_{\Sigma_0} + \tilde{d}_{\Sigma_1}$ on the domain $\tilde{\Omega} := \overline{\Omega} \cap \tilde{M}$ which is the closure of $\Omega$ in $\tilde{M}$. Our assumption that the $M$-closures $\overline{\Sigma}_i$ are disjoint implies that $\tilde{u}(x) \to \infty$ as $\tilde{d}(o', x) \to \infty$. Indeed, suppose $x_i$ is a sequence of points in $\Omega$ such that $x_i \to \bar{x} \in \partial \Omega \cap \partial \tilde{M} \subset M'$. By assumption $\min\{d_{\Sigma_0}(\bar{x}), d_{\Sigma_1}(\bar{x})\} > 0$. Supposing without loss of generality  $d_{\Sigma_0}(\bar{x}) > 0$, given any point $\bar{y} \in \Sigma_0$ and any $g$-unit-speed curve $\gamma$ between $\bar{y}$ and $\bar{x}$, one has 
\[
L(\gamma, \tilde{g}) \geq \int_0^{d(\bar{x}, \bar{y})} (\sn\rho(\gamma(t)))^{-1} \, dt \geq \int_0^{\rho(\bar{y})} (\sn \rho)^{-1} d\rho = \infty.
\]

Thus $\tilde{u} : \tilde{\Omega} \to (0, \infty)$ is a continuous function with a positive lower bound and tends to infinity at infinity. Therefore $\tilde{u}$ must attain its minimum somewhere on $\tilde{\Omega}$. We now apply Proposition \ref{prop:d-maximum-principle} (to $(\tilde{M}, \tilde{g}, e^{-f})$), which implies that both $\Sigma_0$ and $\Sigma_1$ must be totally geodesic with respect to $e^{-\frac{2}{n} f}\tilde{g} = g$. But then (by classification of totally geodesic hypersurfaces in the space form $(M,g)$), we must have $\bar{\Sigma}_0 \cap \bar{\Sigma}_1 \neq \emptyset$. This is the desired conclusion, so we conclude that we must have had $\bar{\Sigma}_0 \cap \bar{\Sigma}_1 \neq \emptyset$ in the first place. 

\end{proof}

\begin{remark}
Using the results from Section \ref{sec:gaussian-conformal}, a very similar argument also gives an alternative proof of the improved Frankel property in setting (III) from Section 2 (i.e. Theorem \ref{thm:ball-shrinker-more-general}). 
\end{remark}

\appendix

\section{Mean convex deformations via the conformal group}
\label{sec:conformal-deformations} 

\begin{proposition} 
Suppose $M \in \{\mathbb{H}^{n+1}, \mathbb{R}^{n+1}, \mathbb{S}^{n+1} \}$ is a simply connected space form and $B^{n+1}_R \subset M$ is a geodesic ball of radius $R \in (0, \mathrm{diam}(M))$. Let $H_0 \subset B^{n+1}_R$ be a totally geodesic, free-boundary hypersurface.

Then for all $\epsilon > 0$ small, there exists a free boundary, strictly mean convex hypersurface $H^\epsilon_0$ lying in an $\epsilon$-neighborhood of $H_0$ satisfying $H_0 \cap H^\epsilon_0 = \emptyset$. In particular, $H^\epsilon_0$ can be chosen to lie in either component of $B^{n+1}_R \setminus H_0$ with non-vanishing mean curvature vector pointing away from $H_0$. 
\end{proposition}
\begin{proof}
Fix a simply-connected space form $M$ with constant curvature $\kappa \in \{-1, 0, 1\}$. We model the geodesic ball $B^{n+1}_R$ by letting
\[
B := \{ (x_0, \dots, x_n) : x_0^2 + \cdots + x_n^2  < \bar{R}^2\}
\]
where $\bar{R} = \tn\frac{R}{2}$ and equipping $B$ with the constant curvature metric given by
\[
g_{ij} := \frac{4}{(1+ \kappa |x|^2)^2} \delta_{ij}. 
\]
Let $H_0= \{x_0 = 0\} \subset B$ denote a totally geodesic free boundary disk through the origin.  

We will obtain a free boundary mean-convex deformation of $H_0$ with respect to the metric $g$ by considering the image of $H_0$ under a conformal mapping of $B$. To that end, for $y \in B$, consider the conformal translation $\Phi_y : B \to B$ mapping the origin to $y$ given by
\[
\Phi_y(x) = \bar{R}^{2} \frac{(\bar{R}^2 + 2 \langle x, y\rangle + |x|^2) y + (\bar{R}^2 - |y|^2) x}{\bar{R}^4 + 2\bar{R}^2  \langle x, y \rangle +|x|^2 |y|^2}. 
\]
Note that $\Phi_y(0) = y$, $\Phi_y^{-1} = \Phi_{-y}$, 
\[
|\Phi_y(x)|^2 = \bar{R}^2 \frac{ \bar{R}^2 |x + y|^2}{\bar{R}^4 + 2 \bar{R}^2 \langle x, y\rangle +|x|^2 |y|^2 }, \qquad 
(\Phi_y)^\ast \delta_{ij} =\Big(\frac{\bar{R}^2(\bar{R}^2 - |y|^2)}{\bar{R}^4 + 2 \bar{R}^2 \langle x, y\rangle + |x|^2 |y|^2} \Big)^2 \delta_{ij}. 
\]
Set $H^s_0 := \Phi_s(H_0)$ for $\Phi_s = \Phi_{y}$, $y = (s, 0, \dots, 0)$ and $s > 0$ small. It is clear that $H^s_0$ foliates a neighborhood to one side of $H_0$. Since $\Phi_s$ is conformal and $H_0$ meets $\partial B$ orthogonally, the deformed hypersurface $H^s_0$ continues to meet $\partial B$ orthogonally (with respect to either metric). 

It remains to show that for $s$ sufficiently small, the mean curvature of $H^s_0$ with respect the metric $g$ is positive (taking the unit normal pointing in the $x_0$ direction). The mean curvature of $\Phi_s(H_0)$ with respect to $g$ is the same as the mean curvature of $H_0$ with respect to $\Phi_s^\ast g = e^{2u} \delta$ where 
\[
e^{u(x)}  := \Big(\frac{\bar{R}^2(\bar{R}^2 - s^2)}{\bar{R}^4 + 2 \bar{R}^2 s x_0 + |x|^2 s^2} \Big)\Big( \frac{2}{1+ \kappa |\Phi_s(x)|^2}\Big).
\]
By the formula for the mean curvature under conformal change (see Appendix \ref{sec:conformal-change} below) and the fact that $H_0$ is totally geodesic in the Euclidean metric $\delta_{ij}$, it suffices to show that $\frac{\partial u}{\partial x_0}(x) < 0$ for $x \in H_0$. Observe that $|\Phi_s(x)|^2 = \frac{\bar{R}^4 (|x|^2 + 2 x_0 s+ s^2)}{\bar{R}^4 + 2 \bar{R}^2 x_0 s+ |x|^2 s^2}$ and so, after some simplification, we obtain 
\begin{align*}
e^{u(x)} &=  \frac{2\bar{R}^2(\bar{R}^2 - s^2)}{\bar{R}^4 +2 \bar{R}^2 x_0 s+ |x|^2 s^2+ \kappa\bar{R}^4 (|x|^2 + 2 x_0 s+ s^2)} \\
& = \frac{2}{1 + \kappa |x|^2} - \frac{4( 1 +  \kappa \bar{R}^2)x_0s}{1 + \kappa |x|^2}  + O(s^2). 
\end{align*}
Hence, by differentiating, for any point $x \in H_0$ (so that $x_0 = 0$), we have
\[
e^u \frac{\partial u}{\partial x_0} \bigg|_{x_0 = 0}= -  \frac{4 (1+  \kappa \bar{R}^2)}{1 + \kappa |x|^2} s + O(s^2).
\]
Since $1 + \kappa \bar{R}^2  > 0$, we deduce that mean curvature of $H^s_0$ (away from $H_0$) is everywhere positive (for $s > 0$ sufficiently small). 
\end{proof}

\section{Cutoff Construction}
\label{sec:cutoff}

For the convenience of the reader, we recall the setting of Lemma \ref{lem:cutoff}. We let $(M, g)$ denote a complete $(n+1)$-dimensional Riemannian manifold and $N \subset M$ is an open subset such that $\partial N = D \sqcup S$ is the union of an open smooth subset $S$ and a closed subset $D$. We assume $T \in \mathbb{I}_n(N, \partial N)$ is an $n$-current which is stationary relative to $D$. The regular set $\Sigma = \reg T$ (consisting of both interior regular points and boundary regular points along $S$) has finite area $\mathcal{H}^n(\Sigma) < \infty$ and the relative singular set $\sing T = \spt T \setminus (\Sigma \cup (\spt \partial T\cap D))$ satisfies $\mathrm{dim}_{\mathcal{H}} \sing T < n -2$. We let $N_\delta := \{x \in N : d(x, D) \geq \delta\}$ denote the set of points in $N$ of distance at least $\delta$ from the fixed-boundary $D$.

\begin{proof}[Proof of Lemma \ref{lem:cutoff}]

The set $\sing T \cap N_\delta$ is the compact subset of (relative) singular points of distance at least $\delta$ to the fixed-boundary of $D$. 

By assumption $\mathrm{dim}_\mathcal{H} \sing T \cap N_\delta < n-2$. In particular, because $S_\delta$ is compact, for any $\bar{r}, \mu > 0$ we can find a finite covering 
\[
\sing T \cap N_\delta \subset \bigcup_{i =1}^P B_{r_i}(x_i)
\]
of $\sing T \cap N_\delta$ by balls $B_{r_i}(x_i)$ with centers $x_i \in N$ of radii $r_i \leq \bar{r}$ such that 
\[
\sum_{i=1}^P r_i^{n-2} < \mu. 
\]
In particular, by taking $\bar{r} < \frac{1}{4}\delta$, we may assume the centers $x_i$ has distance $d(x_i, D) > \frac{3}{4}\delta$ from the boundary $D$.

For each $i$, let $\rho_i := d(x_i, \cdot)$ denote the distance function to the point $x_i$. Next, let us introduce a smooth cutoff function $\eta :[0, \infty) \to [0, 1]$ which satisfies that $\eta = 0$ on $[0, 1]$ and $\eta = 1$ on $[2, \infty)$ with uniform bound $|\eta'| \leq 2$ on the transition $(1, 2)$. Define rescalings by $\eta_r(s) = \eta(s/r)$ which vanish on $[0, r]$, equal 1 on $[2r, \infty)$ and satisfy $|\eta'| \leq 2r^{-1}$.  Now consider functions 
\[
\phi_i(x) := \eta_{r_i}( \rho_i (x))= \eta(\rho_i(x)/r_i).
\] 
The functions $\phi_i$ are Lipschitz continuous, satisfy $\phi_i = 1$ outside $B_{2r_i}(x_i)$, $\phi_i = 0$ inside $B_{r_i}(x_i)$ and have $|\nabla \phi_i| \leq C r_i^{-1}$ in $B_{2r_i}(x_i) \setminus B_{r_i}(x_i)$ wherever the derivative of $\phi_i$ exists (which is for almost every $x$). Here $C$ is a uniform constant depending only upon the dimension. We define our desired cutoff function by 
\[
\phi(x):= \min_{i \in \{1, \dots, P\}} \phi_i(x).
\] 

By construction 
\[
\phi (x) = \begin{cases}
  0 & \text{ if } x \in \bigcup_i B_{r_i}(x_i)\\ 
1 & \text{ if } x \in N \setminus \bigcup_i B_{2r_i}(x_i)
\end{cases}. 
\]
Moreover, as the minimum of finitely many Lipschitz functions, $\phi$ is Lipschitz. The derivative $|\nabla \phi|$ satisfies (almost everywhere)
\[ |\nabla \phi|^2 \leq \max_i |\nabla \phi_i|^2 \leq C^2\max_i r_i^{-2} \mathbf{1}_{B_{2r_i}(x_i) \setminus B_{r_i}(x_i)} \leq C^2\sum_i r_i^{-2} \mathbf{1}_{B_{2r_i}(x_i) \setminus B_{r_i}(x_i)}. \]

In particular $\phi : N_\delta \to [0, 1]$ Lipschitz and it vanishes on $\bigcup_i B_{r_i}(x_i) \supset \sing T \cap N_\delta$ by construction. To obtain the integral estimate, we have
\[
\int_{\Sigma \cap N_\delta}  |\nabla^{\Sigma} \phi|^2 \leq C \sum_{i =1}^{P}  \int_{\Sigma \cap B_{2r_i}(x_i) \setminus B_{r_i}(x_i)}  r_i^{-2}  \leq C \sum_{i =1}^P \mathcal{H}^{n}\big(\Sigma \cap B_{2r_i}(x_i)\big) r_i^{-2}.
\]
Because $T$ is stationary, the monotonicity formula (cf. \cite{LZ21b} Theorem 2.1 for the free boundary case if $x_i \in S$) together with the assumption that $\mathcal{H}^{n-1}(\Sigma) <\infty$ implies a bound $\mathcal{H}^{n}\big(\Sigma \cap B_{2r}(x_i)\big) \leq C r^n$ for any $r \leq \min\{ r(\delta), 1\}$ and any $x \in N_\delta$. Here $r(\delta)$ is a constant depending upon $\delta$ ensuring $r$ is small enough so that $B_{2r}(x_i)$ avoids the fixed boundary $D$. The constant $C$ depends upon the geometry of $N$, the bound for $\mathcal{H}^n(\Sigma)$, and $\delta$, but not upon the radius $r$ once it is less than $\min\{r(\delta),1\}$.

Thus, assuming $\bar{r}$ from above is small enough (depending upon $\delta$) we obtain 
\[
\int_{\Sigma \cap N_\delta} |\nabla^{\Sigma} \phi|^2 \leq C \sum_{i =1}^P  r_i^{n-2} \leq C \mu. 
\]
The result now follows by taking $\mu$ sufficiently small so that $C \mu \leq \varepsilon$. 

\end{proof}

\section{Strong maximum principle for stationary varifolds} 
\label{sec:white}

In \cite{White10}, White proved a strong maximum principle for stationary varifolds, building on earlier joint work of White's with Solomon \cite{SW89}. In \cite{LZ21}, Li and Zhou proved a version of White's maximum principle that works in the free boundary setting. For the convenience of the reader, we restate a version of these results suitable to our purposes here, but refer the reader to the papers for details. 

Continuing with our conventions for $M, N, S, D$. Let $\Omega \subset N$ be a domain such that $\partial \overline{\Omega}$ is smooth almost everywhere. Let $\eta_{\partial \Omega}$ denote the inward-pointing unit normal at any smooth point.   

It is convenient to introduce the notion of minimising to first order; when $S=\emptyset$ we have the simple definition: A varifold $V$ in $\overline{\Omega}$ minimises area to first order in $(\Omega, D)$ provided that
\begin{equation}\label{eq:first-order-minimization}
\delta V(X) := \int \mathrm{div}_V (X) \, d\mu_V \geq 0 
\end{equation}
for all $C^1$ vector fields $X$ with compact support in $\overline{N} \setminus D$ such that $\langle X, \eta_{\partial \Omega} \rangle \geq 0$ at any smooth point in $\partial \Omega$. (These generate inward variations of $\Omega$.)

When $S$ may be nonempty, we restrict to variations which are also tangent along $S$: Let $\eta_S$ denote the inward-pointing unit normal for $S$ in $\overline{N}$. Then a varifold $V$ in $\overline{\Omega}$ minimises area to first order in $(\Omega\cup S, D)$ provided that \eqref{eq:first-order-minimization} holds for all $C^1$ vector fields $X$ with compact support in $\overline{N} \setminus D$ such that $\langle X, \eta_S \rangle = 0$ at any point in $S$ and $\langle X, \eta_{\partial \Omega} \rangle \geq 0$ at any smooth point in $\partial \Omega$.

Finally, let $\Sigma \subset \partial \Omega$ be a smooth component of the boundary of $\Omega$ with the property that $\Sigma$ is properly embedded in $N$ and, if $S$ is nonnempty, has orthogonal intersection along $S$ (as in Definition \ref{def:half-fb} (ii)). Assume $\Sigma$ is (weakly) mean convex (with respect to $\Omega$). That is, $\langle \vec{H}_\Sigma, \eta _\Sigma \rangle \geq 0$, where $\eta_\Sigma$ is the unit normal on $\Sigma$ which points into $\Omega$, and $\vec{H}_\Sigma$ is the mean curvature vector. 

\begin{theorem}[\cite{White10}, \cite{LZ21}]
\label{thm:maximum-principle}
Let $M, N, S, D,\Sigma, \Omega$ be as above. Suppose that $V$ is an integral varifold that minimises area to first order in $(\Omega \cup S, D)$. 

If either $\spt V \cap \Sigma \neq \emptyset$ or $\spt V\cap \partial \Sigma\cap S \neq \emptyset$, then in fact $\Sigma \subset \spt V$ and $\vec{H}_{\Sigma}=0$. Moreover, $V$ may be written as $W+W'$, where $\spt W = \Sigma$, and $\spt W' \cap \Sigma=\emptyset$. 
\end{theorem}

\section{Conformal change formulae}
\label{sec:conformal-change}

Suppose $(M, g)$ is an $(n+1)$-dimensional Riemannian manifold and $\tilde{g} = e^{2\phi} g$ is a metric conformal to $g$. Then the curvature of the conformal metric is given by 
\[ \tilde{R}_{ijkl} = e^{2\phi}(R_{ijkl}- g_{ik} \Phi_{jl} - g_{jl} \Phi_{ik} + g_{il}\Phi_{jk} + g_{jk} \Phi_{il}) ,\]
where 
\[\Phi= \nabla^2\phi - d\phi \otimes d\phi + \frac{1}{2}|d\phi|_g^2 g.\]
The formula for the Ricci curvature and the Hessian of a function $h$ are given by
\[\wt{\Ric} = \Ric - (n-1)(\nabla^2 \phi - d\phi\otimes d\phi) - (\Lap\phi + (n-1)|d\phi|_g^2)g,\]
and
\[ \wt{\nabla}^2 h = \nabla^2 h  - d\phi \otimes dh - dh\otimes d\phi + g(\nabla \phi, \nabla h) \cdot g.\]
If $\Sigma \subset M$ is a hypersurface with normal $\nu$, then its second fundamental form $A_{\Sigma}(X, Y) := g( \nabla_X Y, \nu )$ under conformal change is given by  
\[\wt{A}_\Sigma = \e^{\phi}\left( A_\Sigma - g(\nabla \phi, \nu) g \right).\] 
In particular
\[
|\wt{A}_{\Sigma}|_{\tilde{g}}^2 =  e^{-2\phi} \big(|A_{\Sigma}|_g^2 - 2 g(\nabla \phi, \nu) H_{\Sigma} + n g(\nabla \phi, \nu)^2\big)
\]
and 
\[
\tilde{H}_{\Sigma} = e^{-\phi} \big(H_{\Sigma} - n g(\nabla \phi, \nu)\big).
\]

\bibliographystyle{amsalpha}
\bibliography{improved-frankels}

\end{document}